\documentclass[hidelinks,onefignum,onetabnum]{siamart251216}
\usepackage{amsfonts,amssymb}
\usepackage{graphicx}
\usepackage{booktabs}
\graphicspath{{./}}

\newsiamthm{thm}{Theorem}
\newsiamthm{lem}{Lemma}
\newsiamthm{prop}{Proposition}
\newsiamremark{rmk}{Remark}
\headers{A PDE-theory-guided neural operator}{A. Moustafa, J. Wu, and X. Zheng}

\title{A PDE-theory-guided neural operator and the long-time relaxation of the
two-dimensional anisotropic Boussinesq system\thanks{Submitted to the editors DATE.
\funding{J. Wu was partially supported by the National Science Foundation (NSF) under
grants DMS-2104682 and DMS-2309748. X. Zheng was partially supported by the NSF
under grant DMS-2309747.}}}

\author{Akram Moustafa\thanks{Department of Computer Science, Central Michigan University, Mount Pleasant, MI 48858
  (\email{moust1am@cmich.edu}).}
\and Jiahong Wu\thanks{Department of Mathematics, University of Notre Dame, Notre Dame, IN 46556
  (\email{jwu29@nd.edu}).}
\and Xiaoming Zheng\thanks{Department of Mathematics, Central Michigan University, Mount Pleasant, MI 48858
  (\email{zheng1x@cmich.edu}).}}

\ifpdf
\hypersetup{
  pdftitle={A PDE-theory-guided neural operator and the long-time relaxation of the two-dimensional anisotropic Boussinesq system},
  pdfauthor={A. Moustafa, J. Wu, and X. Zheng}
}
\fi

\begin{document}

\maketitle

\begin{abstract}
We design a neural operator for the two-dimensional Boussinesq system with
anisotropic dissipation on the periodic box, guided by a rigorous analysis of its
long-time behavior. The operator maps a small perturbation of hydrostatic
equilibrium, of amplitude $\varepsilon$, to the vorticity at a late time. We
prove that the flow relaxes to a stratified shear flow, with the horizontally
averaged vorticity staying within $O(\varepsilon^2)$ of its initial value
$\bar\omega_0$ and the oscillation about it of order $\varepsilon$. Accordingly,
the network outputs $\bar\omega_0$ plus these two parts, each scaled to its
size, which grows as $\varepsilon^2$ and $\varepsilon$, respectively. Without this guidance the training either does not
converge or misses the $\varepsilon^2$ growth. With it, one network reproduces,
sample by sample on $2\times10^3$ test runs over a factor of $23$ in amplitude,
the quantity bounded by the theorem to a median error of $2.2$ percent. Being
fast and differentiable with respect to the initial data, the network is also used to search the space of initial
data, and with every candidate checked by the solver the search finds initial
data for which this quantity is $4.4$ times the largest value among the test
runs, raising the lower bound for the constant of the theorem by that factor.
\end{abstract}



\section{Introduction}
\label{sec_intro}

Neural operators approximate the solution operator of a partial
differential equation, that is, the map from its input functions, such as initial and boundary data, coefficients and forcing, and its physical parameters to its solution
\cite{Kovachki2022}. Widely used examples are the Fourier neural operator (FNO)
\cite{Li2021FNO} and the deep operator network (DeepONet) \cite{Lu2021DeepONet},
and both can approximate any continuous operator on a compact set of inputs
\cite{KovachkiLanthalerMishra2021,Lanthaler2022}. Once trained on solutions computed by a conventional solver, a neural operator is
evaluated orders of magnitude faster than the solver, which makes it useful
whenever a solution operator must be evaluated many times. Fluid dynamics is a
main field of application. Fourier neural operators predict turbulent solutions
of the two-dimensional Navier--Stokes equations \cite{Li2021FNO}, reproduce the
long-time statistics of chaotic dissipative systems \cite{Li2022Chaotic}, and
forecast the global weather \cite{Pathak2022FourCastNet}. Closest to the present
problem, they have been trained on two-dimensional Rayleigh--B\'enard convection,
a Boussinesq flow driven by buoyancy, where they predict the velocity and
temperature fields more accurately than earlier data-driven models
\cite{Straat2025}. Physics-informed neural networks \cite{Raissi2019PINN} and
their operator versions \cite{Wang2021PIDeepONet,Li2024PINO} add the residual of
the equation to the training loss, so that the network learns from the equation
as well as from solutions.

Machine learning has also begun to assist mathematical analysis. It has suggested
conjectures and guided their proofs in pure mathematics \cite{Davies2021}, and it
has constructed examples that disprove conjectures or improve known bounds
\cite{Wagner2021,RomeraParedes2024}. For the equations of fluid motion, neural
networks have found an approximate self-similar blow-up profile of the
three-dimensional axisymmetric Euler equations, which also solves the closely
related two-dimensional Boussinesq equations
\cite{WangLaiGomezSerranoBuckmaster2023}, and new families of unstable
singularities \cite{Wang2025Unstable}, accurate enough to serve as a starting
point for computer-assisted proofs. In these works the network proposes
candidates and a separate rigorous or high-precision computation decides whether
they are valid. The same division of labor underlies multifidelity methods, in
which a cheap model explores and an accurate model confirms
\cite{Peherstorfer2018}. In the stability theory of fluids, the search for the
initial data that maximize the growth of a perturbation over a finite time is an
optimization problem, each step of which requires solving the equations forward
and backward in time \cite{PringleKerswell2010,Kerswell2018}.

Known structure can also be built into the network itself. Conservation laws
can be enforced by construction \cite{RichterPowell2022,Hansen2023}, and
asymptotic-preserving networks and neural operators build the limit of a small
parameter of the equation into the loss, so that one network stays accurate
across the range of the parameter \cite{JinMaWu2024,WuYanJinMa2023,ZhangJin2024}.
Asymptotics-informed operator networks add boundary-layer correctors from
singular perturbation theory to the output \cite{LeeHongKoLee2025}.
In these approaches the network learns from solutions computed by a solver, from
the residual of the equation, or from both, and the structure built into it comes
from the equation or from a formal limit. What rigorous analysis proves about the solutions,
such as which quantities are nearly conserved, that is, change only at higher
order in the amplitude, how large the various parts of the solution are, and which quantities depend on the data in an ill-conditioned way, is used
less systematically in the design of the network. In this paper we design a
neural operator from the rigorous long-time analysis of a partial differential
equation, the two-dimensional Boussinesq system with anisotropic dissipation, and
use it to obtain information that the analysis alone does not give. The analysis
decides what the network predicts, how its output is split, and how each part is
scaled with the amplitude of the data. The trained network then takes the place
of the solver in a search over initial data for a lower bound on the constant of
the analytical estimate, and the solver confirms the candidates it finds. The
remainder of this introduction describes the physical problem and the analytical
result on which the design rests.

Buoyancy-driven flows in the Boussinesq approximation are a canonical model of
stratified fluids in the atmosphere and ocean, where the interplay of velocity
and temperature (or density) governs the onset of convection, the persistence of
shear flows, and the emergence of stratified states
\cite{Pedlosky1987,Cushman2011}. In many geophysical regimes the dissipation is
strongly \emph{anisotropic}: turbulent transport and rotation make the effective
diffusion in the horizontal and vertical directions markedly different, and it is
natural, both physically and mathematically, to study models in which the
dissipation is present in only one direction of each field. The system studied
here retains horizontal viscosity in the momentum equation and vertical thermal
diffusion in the temperature equation, and no dissipation in the complementary
directions.

The stability of such partially dissipated Boussinesq systems near the hydrostatic
equilibrium has been established in a sequence of works: for the strip
$\mathbb T\times\mathbb R$ with horizontal velocity dissipation and vertical
thermal diffusion \cite{Adhikari2022}, for related partial-dissipation
configurations \cite{LAI2021764}, and for the buoyancy-driven case without any
thermal diffusion \cite{TaoWuZhaoZheng2020,Doering2018}. A recurring theme is that
the buoyancy coupling itself is stabilizing: the exchange between the vertical
velocity and the temperature converts the momentum equation, which is damped only
in the horizontal direction, into
a damped-wave-type system, so that small perturbations remain globally regular and
uniformly bounded. On the periodic box $\mathbb T^2$, which is the setting of the
present paper, the corresponding uniform $H^2$ stability estimate has become
available only recently \cite{GaoWuZheng2026}. The energy method used on the strip
in \cite{Adhikari2022} does not carry over to the box. On a bounded vertical
interval the one-dimensional Sobolev inequality acquires an additional
zeroth-order term, which breaks the anisotropic product estimates on which that
method rests, and a different mechanism, special to the box, is needed. This estimate is recalled in
Appendix~\ref{app:stability}.

What these stability results do not by themselves settle is the
\emph{long-time limit}: given that a small perturbation exists and remains small
for all time, what does it converge to, and how much of the limit is determined by the
initial data?

 If the dissipation is isotropic and present
in both fields, as in the full-Laplacian system
$\partial_t\mathbf u+\mathbf u\cdot\nabla\mathbf u=-\nabla p+\nu\Delta\mathbf u+\theta\mathbf e_2$,
$\partial_t\theta+\mathbf u\cdot\nabla\theta+u_2=\eta\Delta\theta$, the answer on
$\mathbb T^2$ is immediate: the two coupling terms cancel in the $L^2$ energy
balance, leaving
\[
\tfrac{d}{dt}\big(\|\mathbf u\|_{L^2}^2+\|\theta\|_{L^2}^2\big)
=-2\nu\|\nabla\mathbf u\|_{L^2}^2-2\eta\|\nabla\theta\|_{L^2}^2,
\]
and, provided the conserved mean $\int_{\mathbb T^2}u_1$ also vanishes initially,
the Poincar\'e inequality on zero-mean fields forces the entire perturbation to
decay exponentially to zero, leaving no persistent structure. For the anisotropic
system studied here, with dissipation $\nu\partial_{11}\mathbf u$ and
$\eta\partial_{22}\theta$ only, the same computation gives
\[
\tfrac{d}{dt}\big(\|\mathbf u\|_{L^2}^2+\|\theta\|_{L^2}^2\big)
=-2\nu\|\partial_1\mathbf u\|_{L^2}^2-2\eta\|\partial_2\theta\|_{L^2}^2 .
\]
The right-hand side vanishes on the velocity modes with $k_1=0$ and on the
temperature modes with $k_2=0$, so the Poincar\'e inequality no longer closes the
estimate and the buoyancy coupling itself must supply the stabilization. The closest precedent on the same domain is
\cite{TaoWuZhaoZheng2020} (see also \cite{Doering2018}), which treats the two-dimensional
Boussinesq system on $\mathbb T^2$ with \emph{full} (isotropic) velocity viscosity but
\emph{no} thermal diffusion, and shows that the velocity decays to zero while the
temperature tends to an explicit function of $x_2$ alone. For the system with
full dissipation in both fields, a second-order numerical scheme has been
developed and analyzed in \cite{AlhomsiWuZheng2026}, and its long-time
simulations show the perturbation relaxing exponentially to zero. Our system differs from that of \cite{TaoWuZhaoZheng2020} in
two ways. The velocity dissipation is made anisotropic (horizontal only) and a
vertical thermal diffusion is added. The roles are then
reversed: it is now the \emph{temperature} that decays to zero and the
\emph{vorticity} that tends to a stratified limit, that is, a function of $x_2$
alone (see Theorem~\ref{thm:main}).

The design of our neural operator is guided by the long-time behavior of the
system on the periodic box $\mathbb T^2=[0,2\pi]^2$, which we establish in
Theorem~\ref{thm:main}. Writing any field as its $x_1$-average $\bar g(x_2,t)$
plus its zero-average oscillation $\widetilde g=g-\bar g$, the analysis gives the
following picture. The horizontal viscosity, together with the buoyancy coupling,
damps the oscillatory (that is, $x_1$-dependent) part of the flow: the
oscillation of the vorticity decays and the temperature perturbation decays to
zero. What remains is a purely horizontal, stratified shear flow
$\mathbf u=(\bar u_1(x_2),0)$ carrying the horizontally averaged vorticity
$\bar\omega=-\partial_2\bar u_1$. The averaged vorticity is not conserved, since
the nonlinear advection contributes the term
$\partial_{22}\overline{\widetilde u_1\widetilde u_2}$ to its evolution, where
$\overline{\widetilde u_1\widetilde u_2}$ is the Reynolds stress. However, the
oscillation is of the order $\varepsilon$ of the perturbation and square
integrable in time, and the Reynolds stress is quadratic in it, so the change
$\bar\omega(\cdot,t)-\bar\omega_0$ of the horizontal average is
$O(\varepsilon^2)$ for all time. For the neural operator, this picture tells what
to predict and how large each part is. The final vorticity is the known initial
average $\bar\omega_0$ plus a change of the average of order $\varepsilon^2$ and
an oscillation of order $\varepsilon$, and the temperature, which decays to zero,
is not needed as an output.

The anisotropy makes the solution
operator difficult to learn with a standard neural operator, for three reasons.
First, the quantity bounded by the analysis is small. The change of the horizontal
average of the vorticity is of order $\varepsilon^2$, while the vorticity is of
order $\varepsilon$, and in our data this change is three to four orders of
magnitude below $\bar\omega_0$ in $L^2$. A network that predicts the final
vorticity as a whole would hide the change inside its own error. Second, the
parts of the solution grow at different rates with the amplitude, so the targets
at the smallest and the largest amplitudes differ greatly in size. Third, since
each field is dissipated in only one direction, the temperature modes independent
of $x_2$ are damped by factors that differ by more than ten orders of magnitude
between small and large horizontal wavenumbers. The late-time temperature
therefore depends on the initial data in a severely ill-conditioned way, and in
a preliminary experiment a network that also predicted it did not converge. The analysis
resolves each of these difficulties. We train a Fourier neural operator on $10^4$
initial-to-final pairs generated by a pseudo-spectral solver at ten amplitudes,
predict only the vorticity, and give the network the structure that the analysis
identifies: its output is the initial horizontal average $\bar\omega_0$ plus a
learned approximation of $\omega(\cdot,T)-\bar\omega_0$, split into its
horizontal average and its oscillation, each divided by its size, which grows
as $\varepsilon^2$ and $\varepsilon$, respectively. We then test the bound of the
analysis directly, on $2\times10^3$ further
solver runs and on the surrogate's predictions for the same initial data.

The remainder of the paper is organized as follows. Section~\ref{sec_governing}
sets up the governing equations, the hydrostatic equilibrium, and the
mean/oscillation decomposition. Section~\ref{sec_asymptotic} states the standing
hypotheses, the main result, Theorem~\ref{thm:main}, on the asymptotic
stratification and its limit profile, and the linear decay rate,
Proposition~\ref{prop:linear}.
Section~\ref{sec_NO} introduces
the Fourier neural operator (Sec.~\ref{sec_NOintro}), describes the training data
(Sec.~\ref{sec_trainingdata}), explains why the late-time temperature is not a
prediction target (Sec.~\ref{sec_blade}), sets out the network and its training
(Sec.~\ref{sec_NOtrain}), reports the accuracy and cost of the surrogate
(Sec.~\ref{sec_no_accuracy}), tests Theorem~\ref{thm:main}(b) on solver runs
and surrogate predictions (Sec.~\ref{sec_thm2_test}), compares the surrogate with networks trained without the analysis (Sec.~\ref{sec_baseline}), and uses the surrogate to
search for a lower bound for the constant of Theorem~\ref{thm:main}(b)
(Sec.~\ref{sec_search}).
Section~\ref{sec_discussion} summarizes the results.
Appendix~\ref{app:stability} states the uniform stability estimate and derives
from it the bound used in the proof,
Appendix~\ref{app:thm} contains the proof of Theorem~\ref{thm:main}, and
Appendix~\ref{app:linear} the proof of Proposition~\ref{prop:linear}.

\section{Governing equations and the mean/oscillation decomposition}
\label{sec_governing}
\subsection{The perturbation system}
We consider the two-dimensional anisotropic Boussinesq system from \cite{Adhikari2022}, with the
domain taken to be the square box $\Omega=[0,2\pi]^2$.
Denote the velocity vector as $U$,  the pressure as $P$,
and the temperature as $\Theta$.
The governing equations  are
\begin{eqnarray}
\partial_t U + U\cdot \nabla U &=& -\nabla P + \nu \partial_{11} U + \Theta \mathbf e_2,
\quad x=(x_1,x_2)\in\Omega,
\quad t>0,
\nonumber\\
\partial_t \Theta + U\cdot \nabla \Theta &=& \eta \partial_{22}\Theta,
\nonumber\\
\nabla\cdot U &=& 0,
\nonumber
\end{eqnarray}
where $\nu$ is the kinematic viscosity and $\eta$ the thermal diffusivity.
Here $\mathbf e_2$ is the unit vector in the vertical direction.

Denote a hydrostatic equilibrium as $(U_{he}, P_{he},\Theta_{he})=(0,\frac{1}{2}x^2_2, x_2)$ and the perturbation as $(\mathbf u, p, \theta)=(U-U_{he}, P-P_{he}, \Theta-\Theta_{he})$. We assume the perturbation to be periodic in both spatial directions with period $2\pi$. That is, the domain of the perturbation is the periodic box $\mathbb T^2=[0,2\pi]^2$.

The governing equations are, with the notation $\mathbf u=(u_1, u_2)$,
\begin{eqnarray}
\partial_t \mathbf u + \mathbf u\cdot\nabla \mathbf u &=& -\nabla p
+ \nu \partial_{11} \mathbf u + \theta \mathbf e_2,
\label{per001}
\\
\partial_t \theta + \mathbf u\cdot \nabla \theta + u_2
&=& \eta \partial_{22} \theta,
\label{per002}
\\
\nabla\cdot \mathbf u&=& 0.
\label{per003}
\end{eqnarray}
An equivalent formulation uses the vorticity $\omega=\nabla\times \mathbf u=\partial_1 u_2 - \partial_2 u_1$ and the stream function $\psi$,
\begin{eqnarray}
\partial_t \omega + \mathbf u\cdot \nabla\omega &=&
\nu \partial_{11} \omega + \partial_1 \theta,\label{vort_omega}\\
-\Delta \psi &=& \omega ,\label{psi_poisson}\\
\mathbf u &=& \nabla^{\perp} \psi = (\partial_2 \psi, -\partial_1 \psi),\label{u_psi}\\
\partial_t \theta + \mathbf u\cdot \nabla \theta + u_2
&=& \eta \partial_{22} \theta. \label{vort_theta}
\end{eqnarray}


\subsection{Mean/oscillation decomposition}
\label{sec_meanosc}

 For a field $g(x_1,x_2,t)$ that is periodic in
$x_1$ we write
\begin{equation}
\bar g(x_2,t):=\frac1{2\pi}\int_0^{2\pi} g(x_1,x_2,t)\,dx_1,
\qquad
\widetilde g:=g-\bar g,
\label{avg_def}
\end{equation}
so that $\bar g$ is independent of $x_1$, and $\widetilde g$ has zero $x_1$-mean for
every $x_2$. A field independent of $x_1$ is called stratified. The vertical counterpart of \eqref{avg_def},
\begin{equation}
Q_0g(x_1,t):=\frac1{2\pi}\int_0^{2\pi} g(x_1,x_2,t)\,dx_2,
\qquad
g^{\perp}:=g-Q_0g,
\label{vert_avg}
\end{equation}
enters only in Appendix~\ref{app:stability}.

The two decompositions separate what each dissipative term can reach. The
horizontal viscosity $\nu\partial_{11}$ annihilates the $x_1$-average and acts as
$-\nu k_1^2$ on the oscillation, so it damps $\widetilde g$ and leaves $\bar g$
untouched. By
contrast, the vertical diffusion $\eta\partial_{22}$ annihilates $Q_0g$ and damps
$g^{\perp}$. 

In the analysis we write the perturbation velocity $\mathbf u$ of
\eqref{per001}--\eqref{per003} as $u$, and its initial value as $u_0$.
 Each property in Lemma~\ref{lem:meanosc} is a direct consequence of the definitions, the periodicity of the box and, where
indicated, the incompressibility of $u$.

\begin{lem}[Mean/oscillation calculus]
\label{lem:meanosc}
Let $\bar g$ and $\widetilde g=g-\bar g$ be the horizontal mean and oscillation
defined in \eqref{avg_def}, and let $u=(u_1,u_2)$ be a $2\pi$-periodic,
divergence-free vector field with vorticity
$\omega=\partial_1 u_2-\partial_2 u_1$. For $f,g\in H^1(\mathbb T^2)$, and with
$X$ denoting $L^2$ or any of the norms of \eqref{norm_def} below, whenever the
norms involved are finite:
\begin{itemize}
\item[\rm(a)] $\overline{\bar g}=\bar g$,\quad $\overline{\widetilde g}=0$,\quad
  $\overline{\bar f\,\widetilde g}=\bar f\,\overline{\widetilde g}=0$,\quad
  $\overline{fg}=\bar f\,\bar g+\overline{\widetilde f\,\widetilde g}$.
\item[\rm(b)] $\overline{\partial_2 g}=\partial_2\bar g$,\quad
  $\widetilde{\partial_2 g}=\partial_2\widetilde g$,\quad
  $\overline{\partial_1 g}=0$,\quad $\partial_1 g=\partial_1\widetilde g$,\quad
  $\bar\omega=-\partial_2\bar u_1$,\quad
  $\widetilde\omega=\partial_1\widetilde u_2-\partial_2\widetilde u_1$.
\item[\rm(c)] $\|g\|_{X}^2=\|\bar g\|_{X}^2+\|\widetilde g\|_{X}^2$,\ \ and
  $\|\bar g\|_{X}\le\|g\|_{X}$, $\|\widetilde g\|_{X}\le\|g\|_{X}$.
\item[\rm(d)] $\partial_2\bar u_2=0$, and if in addition
  $\int_{\mathbb T^2}u_2\,dx=0$, then $\bar u_2\equiv0$ and
  $u_2=\widetilde u_2$.
\item[\rm(e)] {\it (Horizontal Poincar\'e)}\ \ since the lowest nonzero
  horizontal frequency on $[0,2\pi]$ is one,
  \begin{equation}
  \|\widetilde g\|_{L^2}\le\|\partial_1\widetilde g\|_{L^2}=\|\partial_1 g\|_{L^2},
  \qquad
  \|\partial_2\widetilde g\|_{L^2}=\|\widetilde{\partial_2 g}\|_{L^2}\le\|\partial_2 g\|_{L^2},
  \label{poincare_x}
  \end{equation}
  and hence, the first inequality being applied to $\widetilde g$ and to
  $\partial_2\widetilde g=\widetilde{\partial_2 g}$,
  $\|\widetilde g\|_{H^1}\le\sqrt2\,\|\partial_1 g\|_{H^1}$, in particular
  $\|\widetilde u\|_{H^1}\le\sqrt2\,\|\partial_1 u\|_{H^1}$.
\item[\rm(f)] {\it (Vertical Poincar\'e)}\ \ if $h$ satisfies
  $\int_0^{2\pi}h(x_1,x_2)\,dx_2=0$ for every $x_1$, then
  \begin{equation}
  \|h\|_{L^2}\le\|\partial_2 h\|_{L^2}.
  \label{poincare_y}
  \end{equation}
  
\item[\rm(g)] $\partial_2Q_0g=0$ and $\int_0^{2\pi}g^{\perp}\,dx_2=0$ for every
  $x_1$, so the decomposition $g=Q_0g+g^{\perp}$ of \eqref{vert_avg} is
  orthogonal in $L^2(\mathbb T^2)$, and $\partial_1$ commutes with both
  projections.
\end{itemize}
\end{lem}

\section{Main result: stratification and the limit profile}
\label{sec_asymptotic}

\subsection{Standing hypotheses on the initial data}
\label{sec_hyp}

Throughout, the initial data $(u_0,\theta_0)\in H^2(\mathbb T^2)$ are assumed to satisfy $\nabla\cdot u_0=0$ and 
\begin{equation}
\int_{\mathbb T^2}\theta_0\,dx=\int_{\mathbb T^2}u_{2,0}\,dx
=\int_{\mathbb T^2}u_{1,0}\,dx=0,
\qquad
\|u_0\|_{H^2}+\|\theta_0\|_{H^2}\le\varepsilon .
\label{data_hyp}
\end{equation}
Here and below, with the
Fourier coefficients $\hat f_k=(2\pi)^{-2}\int_{\mathbb T^2}f(x)e^{-ik\cdot x}\,dx$,
\begin{equation}
\|f\|_{H^{s}}^2=(2\pi)^2\sum_{k\in\mathbb Z^2}(1+|k|^2)^{s}|\hat f_k|^2,
\qquad s\in\mathbb R,
\label{norm_def}
\end{equation}
summed over the components of a vector field.

The mean conditions in \eqref{data_hyp} are preserved by the flow. Integrating
\eqref{per001}--\eqref{per002} over $\mathbb T^2$, the pressure and transport
terms vanish by periodicity and $\nabla\cdot u=0$, and
\begin{equation*}
\frac{d}{dt}\int_{\mathbb T^2}u_1\,dx=0,\qquad
\frac{d}{dt}\int_{\mathbb T^2}\theta\,dx=-\int_{\mathbb T^2}u_2\,dx,\qquad
\frac{d}{dt}\int_{\mathbb T^2}u_2\,dx=\int_{\mathbb T^2}\theta\,dx.
\end{equation*}
Thus $\int_{\mathbb T^2}u_1\,dx$ is conserved, and
\begin{align*}
\int_{\mathbb T^2}\theta\,dx&=\cos t\int_{\mathbb T^2}\theta_0\,dx-\sin t\int_{\mathbb T^2}u_{2,0}\,dx,\\
\int_{\mathbb T^2}u_2\,dx&=\sin t\int_{\mathbb T^2}\theta_0\,dx+\cos t\int_{\mathbb T^2}u_{2,0}\,dx,
\end{align*}
which are periodic in time with period $2\pi$ and do not decay unless both
initial means vanish. Under \eqref{data_hyp} all three means are zero for all
$t\ge0$.

\subsection{Statement of the main result}
By Theorem~\ref{thm:stability} of Appendix~\ref{app:stability}, quoted from
\cite{GaoWuZheng2026}, there is $\varepsilon_*=\varepsilon_*(\nu,\eta)>0$ such
that for $\varepsilon\le\varepsilon_*$ the system \eqref{per001}--\eqref{per003}
with initial data satisfying \eqref{data_hyp} has a unique global solution. This
solution is bounded by $C_*\varepsilon$ in $H^2$ for all time, and its
dissipation integral
$\int_0^\infty(\nu\|\partial_1u\|_{H^2}^2+\eta\|\partial_2\theta\|_{H^2}^2)\,d\tau$
is bounded by $C_\nu\varepsilon^2$. Theorem~\ref{thm:main} describes the
long-time behavior of this solution.

Averaging the first component of \eqref{per001} over $x_1$, the terms
$-\partial_1p$ and $\nu\partial_{11}u_1$ drop out by Lemma~\ref{lem:meanosc}(b).
Since $\nabla\cdot u=0$, $u\cdot\nabla u_1=\partial_1(u_1^2)+\partial_2(u_1u_2)$,
and the first term averages to zero. Hence
$\partial_t\bar u_1=-\partial_2\,\overline{u_1u_2}$.
By Lemma~\ref{lem:meanosc}(a),
$\overline{u_1u_2}=\bar u_1\bar u_2+\overline{\widetilde u_1\widetilde u_2}$, and by
the mean conditions of \eqref{data_hyp}, which hold for all $t\ge0$, and
Lemma~\ref{lem:meanosc}(d), $\bar u_2=0$. Therefore
$\partial_t\bar u_1=-\partial_2\,\overline{\widetilde u_1\widetilde u_2}$, and since
$\bar\omega=-\partial_2\bar u_1$ by Lemma~\ref{lem:meanosc}(b),
\begin{equation}
\partial_t\bar\omega=\partial_{22}\,\overline{\widetilde u_1\widetilde u_2}.
\label{mean_eq}
\end{equation}

\begin{thm}[Small-data stratification and the limit profile on $\mathbb T^2$]
\label{thm:main}
Let $\varepsilon_*=\varepsilon_*(\nu,\eta)>0$ be the threshold of
Theorem~\ref{thm:stability}.
Let the initial data satisfy \eqref{data_hyp} with $\varepsilon\le\varepsilon_*$
and let $(u,\theta)$ be the corresponding global solution of
\eqref{per001}--\eqref{per003}. Let $\omega=\partial_1u_2-\partial_2u_1$, and let $\bar\omega$ and
$\bar\omega_0$ be the horizontal averages \eqref{avg_def} of $\omega$ and
$\omega_0$. Then the zero means are
preserved, $\int_{\mathbb T^2}\theta(t)=\int_{\mathbb T^2}u_2(t)=0$ for all
$t\ge0$, and the following hold.
\begin{itemize}
\item[\rm(a)] {\it (Decay of the vertical velocity and the temperature)} As
$t\to\infty$,
\[
u_2(t)\to0
\quad\text{and}\quad
\theta(t)\to0
\qquad\text{in } H^s(\mathbb T^2)\ \text{ for every } s<2 .
\]
\item[\rm(b)] {\it (Stratification and limit profile of the vorticity)} There is
$\bar\omega_\infty=\bar\omega_\infty(x_2)\in H^1(\mathbb T^2)$ such that
$\omega(\cdot,t)\to\bar\omega_\infty$ in $H^s(\mathbb T^2)$ for every $s<1$ as
$t\to\infty$, and
\begin{equation}
\bar\omega_\infty=\bar\omega_0+\int_0^\infty\partial_{22}\,
\overline{\widetilde u_1\widetilde u_2}\,(\cdot,t)\,dt,
\label{limit_rep}
\end{equation}
where the integral converges in $H^{-2}$. Moreover
\[
\sup_{t\ge0}\ \|\bar\omega(\cdot,t)-\bar\omega_0\|_{H^{-2}}\le C\varepsilon^{2}
\qquad\text{and}\qquad
\|\bar\omega_\infty-\bar\omega_0\|_{H^{-2}}\le C\varepsilon^{2},
\]
where $C$ depends only on $\nu$.
\end{itemize}
\end{thm}

The proof is given in Appendix~\ref{app:thm}.
Since $\bar\omega_0$ is determined by $u_0$ alone, Theorem~\ref{thm:main}(b)
shows that the initial temperature affects the limit only at order
$\varepsilon^2$: for a fixed $u_0$, any two $\theta_0$ admissible under
\eqref{data_hyp} give limits $\bar\omega_\infty$ that differ by at most
$2C\varepsilon^2$ in $H^{-2}$. By \eqref{limit_rep},
$\bar\omega_\infty-\bar\omega_0$ is a nonlinear functional of the whole
trajectory, and no formula for it in terms of the initial data is known.
Section~\ref{sec_NO} approximates the map from the initial data to it by a
neural operator.

\subsection{The linear decay rate}

\begin{prop}[Linear decay of $(u_2,\theta)$]
\label{prop:linear}
Consider the linearization of \eqref{vort_omega} and \eqref{per002} about the
hydrostatic state,
\begin{equation}
\partial_t\omega=\nu\partial_{11}\omega+\partial_1\theta,\qquad
\partial_t\theta=\eta\partial_{22}\theta-u_2,
\label{lin_sys}
\end{equation}
with $\mathbf u$ obtained from $\omega$ by the Biot--Savart law \eqref{psi_poisson}
and \eqref{u_psi}. For each $k=(k_1,k_2)\in\mathbb Z^2\setminus\{0\}$ the Fourier
coefficients satisfy
\begin{equation}
\frac{d\hat u_2}{dt}=-\nu k_1^2\,\hat u_2+\frac{k_1^2}{|k|^2}\,\hat\theta,\qquad
\frac{d\hat\theta}{dt}=-\hat u_2-\eta k_2^2\,\hat\theta .
\label{lin_mode}
\end{equation}
\begin{itemize}
\item[(a)] If $k_1\neq0$ and $k_2\neq0$, then
\[
|\hat u_2(t)|^2+\frac{k_1^2}{|k|^2}|\hat\theta(t)|^2\le
e^{-2\min\{\nu,\eta\}t}\Big(|\hat u_2(0)|^2+\frac{k_1^2}{|k|^2}|\hat\theta(0)|^2\Big).
\]
If $k_1=0$ and $k_2\neq0$, then $\hat u_2\equiv0$ and
$\hat\theta(t)=e^{-\eta k_2^2t}\hat\theta(0)$.
\item[(b)] If $k_2=0$ and $\nu k_1^2>2$, then
$\hat\theta(t)=a_+e^{\lambda_+t}+a_-e^{\lambda_-t}$, with constants $a_\pm$
determined by $\hat u_2(0)$ and $\hat\theta(0)$, $\lambda_-\le-\nu k_1^2/2<-1$, and
\[
\lambda_+=\frac{-2}{\nu k_1^2+\sqrt{\nu^2k_1^4-4}},\qquad
\frac{1}{\nu k_1^2}\le|\lambda_+|\le\frac{2}{\nu k_1^2} <1.
\]
In particular $\lambda_+\to0$ as $|k_1|\to\infty$, so the linear decay of $\theta$
has no uniform rate.
\end{itemize}
\end{prop}

The proof is given in Appendix~\ref{app:linear}. The absence of a uniform decay rate in part (b) is consistent with Theorem~\ref{thm:main}, which treats the nonlinear case and also does not give a decay rate.

\section{\texorpdfstring{A neural-operator surrogate, its training data, and a test of Theorem~\ref{thm:main}(b)}{A neural-operator surrogate, its training data, and a test of Theorem 1(b)}}
\label{sec_NO}
\subsection{Neural operators and the Fourier neural operator}
\label{sec_NOintro}

A neural operator (NO) \cite{Kovachki2022} is a deep-learning architecture that
learns a map between two infinite-dimensional function spaces, rather than
between finite-dimensional Euclidean vectors as in a standard neural network.
Given an input Banach space $\mathcal A$ and an output Banach space $\mathcal U$
of functions on a domain $\mathcal D\subset\mathbb R^d$ (here
$\mathcal D=\mathbb T^2$ and $d=2$), the target is a generally
nonlinear operator $\mathcal G^\dagger:\mathcal A\to\mathcal U$, and the neural
operator provides a parametric approximation
$\mathcal G_\phi\approx\mathcal G^\dagger$ trained from a finite set of
input-output pairs $\{(a^{(i)},z^{(i)})\}$. For the present problem this
operator maps the small-amplitude initial perturbation to the state at a later
time: the input is the pair of initial fields $a=(\omega_0,\theta_0)$, at an
amplitude $\varepsilon$ that varies across the data, and the output is the
vorticity $z=\omega(\cdot,T)$ at the final time (Sec.~\ref{sec_trainingdata}).

\paragraph{Architecture.}
A neural operator is a composition of a lifting map, a sequence of $L$ hidden
layers (here $L=4$), and a projection map,
\begin{equation}
\mathcal G_\phi
=\mathcal Q\circ\mathcal H_L\circ\cdots\circ\mathcal H_1\circ\mathcal P . \nonumber
\end{equation}
The pointwise \emph{lifting} operator $\mathcal P$ raises the input function $a$
to a function $v_0$ with more components, called channels, and the pointwise
\emph{projection} operator $\mathcal Q$ maps the last hidden function back to the
output $z$. Each hidden layer updates the current function
$v_\ell:\mathcal D\to\mathbb R^{d_{v_\ell}}$ by
\begin{equation}
v_{\ell+1}(x)=\varrho\!\Big(W_\ell\,v_\ell(x)+\big(\mathcal K_\ell v_\ell\big)(x)+b_\ell(x)\Big),
\qquad x\in \mathcal D,
\label{eq_NOlayer}
\end{equation}
where $\varrho$ is a fixed nonlinear activation applied pointwise, $W_\ell$ is a
local (pointwise) linear transformation, $b_\ell$ is a bias function, and
$\mathcal K_\ell$ is a nonlocal linear \emph{integral kernel operator}
\begin{equation}
\big(\mathcal K_\ell v_\ell\big)(x)=\int_{\mathcal D} \kappa^{(\ell)}(x,x')\,v_\ell(x')\,\mathrm dx' .
\label{eq_kernel}
\end{equation}

\paragraph{Fourier neural operator.}
The kernel \eqref{eq_kernel} admits several parameterizations, and we use the
\emph{Fourier neural operator} (FNO) \cite{Li2021FNO}, which represents the kernel directly in the
frequency domain. Assuming a translation-invariant kernel
$\kappa^{(\ell)}(x,x')=\kappa^{(\ell)}(x-x')$ and applying the convolution theorem turns
\eqref{eq_kernel} into a multiplication in Fourier space,
\begin{equation}
\big(\mathcal K_\ell v_\ell\big)(x)
=\mathcal F^{-1}\!\big(S^{(\ell)}_\phi\cdot\mathcal F(v_\ell)\big)(x),
\label{eq_FNO}
\end{equation}
where $\mathcal F$ and $\mathcal F^{-1}$ are the Fourier transform and its
inverse, and $S^{(\ell)}_\phi$ is a learned complex weight tensor acting on the
Fourier coefficients. Only the lowest wavenumbers, those with $|k_j|$ up to
$k_{\max}$ in each direction, are retained, and for each of them $S^{(\ell)}_\phi$ holds a trainable
complex $d_{v_{\ell+1}}\times d_{v_\ell}$ matrix. The
higher modes are truncated, which both limits the parameter count and acts as a
spectral smoothing filter. For real-valued fields the conjugate symmetry
$S(-k)=S(k)^{*}$ is imposed so that the output remains real. On a uniform
periodic grid $\mathcal F$ is evaluated by the fast Fourier transform, giving a
quasi-linear cost $\mathcal O(N^2\log N)$ on an $N\times N$ grid.

\subsection{Training data generation}
\label{sec_trainingdata}

The training pairs are produced by numerical simulation of the
perturbation system in its vorticity-stream-function form on the periodic box
$\mathbb T^2$, using a Fourier pseudo-spectral discretization in
space and a second-order semi-implicit time integrator.

\paragraph{Spatial discretization.}
All fields are represented by their truncated Fourier series on an $N\times N$
grid with $N=128$. Linear operators are diagonal in Fourier space:
differentiation is multiplication by $ik$, the stream-function solve
$-\Delta\psi=\omega$ is division by $|k|^2$, and the anisotropic dissipation
terms $\nu\partial_{11}\omega$ and $\eta\partial_{22}\theta$ become
$-\nu k_1^2\hat\omega$ and $-\eta k_2^2\hat\theta$. The quadratic advection terms
are evaluated pseudo-spectrally in conservative form,
$\mathbf u\cdot\nabla\omega=\nabla\cdot(\mathbf u\,\omega)$ and
$\mathbf u\cdot\nabla\theta=\nabla\cdot(\mathbf u\,\theta)$ (using
$\nabla\cdot\mathbf u=0$): the products $\mathbf u\,\omega$ and
$\mathbf u\,\theta$ are formed on the grid in physical space and their
divergence is taken in Fourier space.

\paragraph{Time integration.}
Writing the two evolution equations as
$\partial_t\omega=\mathcal N_\omega+\nu\partial_{11}\omega$ and
$\partial_t\theta=\mathcal N_\theta+\eta\partial_{22}\theta$ with the
nonlinear-and-coupling parts
\begin{eqnarray}
\mathcal N_\omega &=& -\,\mathbf u\cdot\nabla\omega+\partial_1\theta,\nonumber\\
\mathcal N_\theta &=& -\,\mathbf u\cdot\nabla\theta-u_2,\nonumber
\end{eqnarray}
the system is advanced with a second-order semi-implicit scheme based on
the backward differentiation formula (BDF2): the linear anisotropic dissipation is treated
implicitly, while the advection and the buoyancy and coupling terms
$\mathcal N_\omega,\mathcal N_\theta$ are treated explicitly by second-order
extrapolation from the two previous time levels. In Fourier space, for each
wavenumber $k=(k_1,k_2)$,
\begin{eqnarray}
\Big(\tfrac{3}{2\Delta t}+\nu k_1^2\Big)\hat\omega^{\,n+1}
&=& \frac{2\hat\omega^{\,n}-\tfrac12\hat\omega^{\,n-1}}{\Delta t}
+\widehat{\mathcal N}_\omega^{\,*},\nonumber\\
\Big(\tfrac{3}{2\Delta t}+\eta k_2^2\Big)\hat\theta^{\,n+1}
&=& \frac{2\hat\theta^{\,n}-\tfrac12\hat\theta^{\,n-1}}{\Delta t}
+\widehat{\mathcal N}_\theta^{\,*},\nonumber
\end{eqnarray}
where $\mathcal N^{*}$ denotes the nonlinear terms evaluated at the extrapolated
state $(\omega^{*},\theta^{*})=(2\omega^{n}-\omega^{n-1},\,
2\theta^{n}-\theta^{n-1})$, with the corresponding velocity obtained from
$\omega^{*}$ by the Biot--Savart law \eqref{psi_poisson} and \eqref{u_psi}. The
multistep scheme is initialized with a single first-order step of the same
semi-implicit form, backward Euler for the dissipation and forward Euler for
$\mathcal N_\omega,\mathcal N_\theta$. No dealiasing filter is applied to the quadratic terms. The Courant--Friedrichs--Lewy (CFL) number $\Delta t\,\|\mathbf u\|_{\infty}/h$, with $h=2\pi/N$, is
monitored at every step and kept below $1/2$.

As is standard for odd-order derivatives on an even grid
\cite[Ch.~3]{Trefethen2000SpectralMI}, the coefficients of the vorticity at the
Nyquist wavenumbers $k_1=N/2$ and $k_2=N/2$ are set to zero, in the initial data
and at every time step.
The simulations use $\nu=\eta=1$, time step $\Delta t=10^{-3}$, and are
integrated to the final time $T=100$.

\paragraph{Initial data and data sets.}
For each sample the initial vorticity $\omega_0$ and initial temperature
$\theta_0$ are sampled independently as zero-mean Gaussian random fields on
$\mathbb{T}^2$ with covariance operator $(\tau^2-\Delta)^{-\alpha}$. That is, the
Fourier coefficients are independent complex Gaussians with variance proportional
to $(\tau^2+|k|^2)^{-\alpha}$. For $\omega_0$ we take $\alpha=2.5$ and $\tau=7$,
the values used for the initial vorticity of the two-dimensional Navier--Stokes
equations in \cite{Kovachki2022}, and for $\theta_0$ we take $\alpha=3.5$ and
$\tau=7$. With probability one a sample lies in $H^s(\mathbb T^2)$ exactly when $s<\alpha-1$, and the
velocity is one derivative smoother than the vorticity, so $\mathbf u_0$ and
$\theta_0$ both lie in $H^s$ for every $s<5/2$, in particular in $H^2$. The
zero-frequency mode is set to zero, so that
$\int_{\mathbb{T}^2}\theta_0\,dx=0$ in accordance with the zero-mean hypothesis
of Theorem~\ref{thm:main}. The velocity is recovered from the vorticity through
the stream function (Biot--Savart law),
\begin{equation}
\hat{\mathbf u}(k)=\Big(\,i\,k_2/|k|^2,\ -\,i\,k_1/|k|^2\,\Big)\,\hat\omega(k),
\qquad k\neq0. \nonumber
\end{equation}
Finally each sample is
rescaled so that the smallness hypothesis \eqref{data_hyp} of
Theorem~\ref{thm:main} is realized with equality, the norm being divided evenly
between the two fields,
\begin{equation}
\|\mathbf u_0\|_{H^2}=\|\theta_0\|_{H^2}=\tfrac12\varepsilon .
\label{train_norm}
\end{equation}

For each sample we store the initial pair
$(\omega_0,\theta_0)$ and the corresponding final fields
$(\omega(\cdot,T),\theta(\cdot,T))$, all on the $128\times128$ grid, together with
its amplitude and the index of its group, a group being the five samples at one
amplitude that share one initial vorticity. Because the rescaling of $\omega_0$
does not involve $\theta_0$, the members of a group share their initial vorticity
exactly. The \emph{training set} uses ten amplitudes, consecutive ones differing by the
factor $2^{1/2}$,
\begin{equation}
\varepsilon_j=0.025\cdot2^{(2j-1)/4},\qquad j=0,\dots,9,
\label{eps_levels}
\end{equation}
with $1000$ samples each, $200$ initial vorticities each paired with five independent
initial temperatures, $10^4$ samples in all. Two test sets are generated with
independent random seeds. The \emph{training-amplitude} set uses the same ten amplitudes with $100$ samples each, $20$ initial vorticities each paired with five
independent initial temperatures. The \emph{interpolation} set uses the five amplitudes
$\varepsilon\in\{0.025,\,0.05,\,0.1,\,0.2,\,0.4\}$ with $200$ samples each, $40$ initial
vorticities each paired with five independent initial temperatures. Each interpolation amplitude lies a factor $2^{1/4}$ from the two nearest
training amplitudes, inside the trained range. Neither test set enters the loss. The $1.2\times10^4$ simulations
took $33.8$ hours of wall-clock time in all, an average of $10.2$\,s each. Each of
the $25$ amplitudes of the three sets was generated as a separate job, the
jobs running in parallel, so the $10^3$ training samples of one amplitude took between
$2.2$ and $3.1$ hours and the complete data were available after about three
hours.

\subsection{The late-time temperature and why it is not a prediction target}
\label{sec_blade}

\paragraph{Linear prediction.}
By Proposition~\ref{prop:linear},
$|\hat\theta(k_1,k_2,t)|$ with $k_2\neq0$ decays at least like $e^{-\min\{\nu,\eta\}t}$, which is $e^{-100}$ at $T=100$ when $\nu=\eta=1$. When $k_2=0$, $|\hat\theta(k_1,0,t)|$ decays like $e^{-t/(\nu k_1^{2})}$, since $|\lambda_+|=(\nu k_1^2)^{-1}+O\big((\nu k_1^2)^{-3}\big)$ by the formula in Proposition~\ref{prop:linear}(b). Let $E_t(k_1)$ denote the expectation of $|\hat\theta(k_1,0,t)|^2/\varepsilon^2$. Then
\begin{equation}
E_T(k_1)\;\approx\;E_0(k_1)\,e^{-2T/(\nu k_1^{2})},
\quad \text{ when } \nu k_1^2\gg1.
\label{blade_filter}
\end{equation}
For the Gaussian fields of Sec.~\ref{sec_trainingdata}, $E_0(k_1)\propto(\tau^2+k_1^2)^{-\alpha}$ with
$\alpha=3.5$ and $\tau=7$, and the right-hand side of \eqref{blade_filter} when $T=100$ and $\nu=1$ is largest at $k_1^{\ast}=9.4$.

\paragraph{Comparison with the simulations.}
The $2000$ spectral simulations of the two test sets agree with this analysis. At
$T=100$ the temperature of each run is, to a close approximation, a function of
$x_1$ alone, oscillating in $x_1$ (Fig.~\ref{fig:bladefield}). The modes with
$k_2=0$ carry a median $99.9999$ percent of $\|\theta(\cdot,T)\|_{L^2}^2$, against
$5$ percent in the initial data, and the $L^2$ norm of $\theta$ has fallen by a
median factor of about $50$. The wavenumber $|k_1|$ at which
$|\hat\theta(k_1,0,T)|^2$ is largest has median $10$, close to $k_1^{\ast}=9.4$.
At every $k_1$ with $7\le k_1\le20$, a range that carries $90$ percent of
$\sum_{k_1}E_T(k_1)$, the right-hand side of \eqref{blade_filter}, with $E_0$
replaced by its average over the $2000$ runs, differs from the average of
$|\hat\theta(k_1,0,T)|^2/\varepsilon^2$ over the same runs by a relative error of
at most $0.6$ percent. The oscillation of $\theta(\cdot,T)$ in $x_1$ is therefore
not a numerical artifact. It is the part of $\theta_0$ with $k_2=0$, from which
the coupling with $u_2$ has removed the modes with small $|k_1|$, while the decay
of $E_0$ with $|k_1|$ keeps the energy at large $|k_1|$ small.

\begin{figure}[htbp]
\centering
\includegraphics[width=0.55\textwidth]{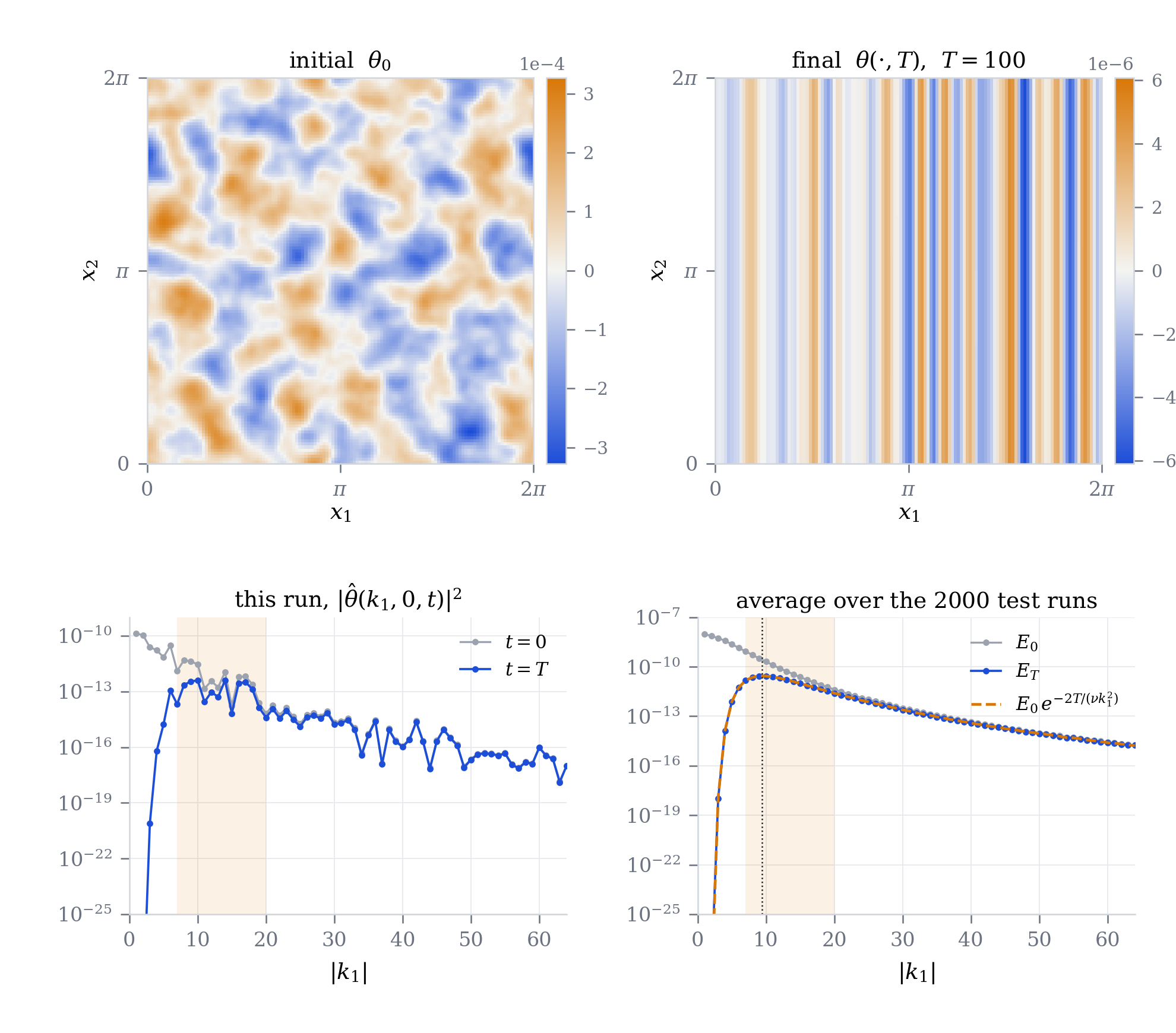}
\caption{Top: the initial temperature and the temperature at $T=100$ of a
representative run of the interpolation set, at $\varepsilon=0.1$ (note the
differing color scales). Bottom left: $|\hat\theta(k_1,0,t)|^2$ of the same run at
$t=0$ and $t=T$. Bottom right: the averages of $|\hat\theta(k_1,0,t)|^2/\varepsilon^2$
over the $2000$ test runs at $t=0$ and $t=T$, which estimate $E_0$ and $E_T$, and the
right-hand side of \eqref{blade_filter} evaluated with the average at $t=0$. The
dotted line marks $k_1^{\ast}=9.4$. In both panels the shaded band is
$7\le|k_1|\le20$.}
\label{fig:bladefield}
\end{figure}

\paragraph{Why the temperature is not a prediction target.}
By \eqref{blade_filter} the map $\theta_0\mapsto\theta(\cdot,T)$ is severely
ill-conditioned. On the modes with $k_2=0$ the measured value of the multiplier,
$(E_T(k_1)/E_0(k_1))^{1/2}$, is $4.8\times10^{-11}$, $1.4\times10^{-5}$,
$1.9\times10^{-3}$, $0.368$, $0.902$ and $1.00$ at $k_1=2,3,4,10,32,64$, the last
because the vertical velocity vanishes at $k_1=N/2=64$, where the coefficients of
the vorticity are set to zero, so that this mode is neither damped nor coupled.
The ratio of its largest to its smallest value exceeds $10^{10}$, and in a
preliminary experiment at a single amplitude a network trained with both
$\omega(\cdot,T)$ and $\theta(\cdot,T)$ as targets did not converge. Moreover, by Theorem~\ref{thm:main}(a) the temperature
decays to zero, and at $T=100$ its median $L^2$ norm is $1.3\times10^{-5}$. The
network therefore predicts only the vorticity, and $\theta_0$ is kept as an input
because it enters the evolution of the vorticity.

\subsection{Network structure and training}
\label{sec_NOtrain}

\paragraph{Network architecture.}
The Fourier neural operator of Sec.~\ref{sec_NOintro} is instantiated in two
spatial dimensions with $L=4$ Fourier layers, $k_{\max}=32$, so that the wavenumbers
$-32\le k_1<32$ and $|k_2|\le32$ are retained, the second direction being
handled by a real Fourier transform, and $64$ hidden channels. The lifting and projection
maps use channel-expansion ratios of $2$ and $4$, respectively. The network has
three input channels and two output channels on the $128\times128$ grid.

The inputs are the rescaled initial fields $\omega_0/\varepsilon$ and
$\theta_0/\varepsilon$, and a constant channel $\log(\varepsilon/0.1)$ that encodes
the amplitude. Each channel is standardized with one mean and one
standard deviation computed over the training set.

The outputs are not the final vorticity itself. Theorem~\ref{thm:main}(b) says that
the horizontal average of the vorticity stays within $C\varepsilon^2$ of
$\bar\omega_0$ at every time. On the test sets
the change of the horizontal average,
\begin{equation}
\delta\bar\omega:=\bar\omega(\cdot,T)-\bar\omega_0 ,
\label{dwbar_def}
\end{equation}
is a median $5.4\times10^{-5}$
to $1.2\times10^{-3}$ of $\bar\omega_0$ in $L^2$,
growing with $\varepsilon$. A network that predicts $\omega(\cdot,T)$ whole would have
to represent a correction three to four orders of magnitude below the field it
outputs. We therefore give the surrogate the structure of the theorem,
\begin{equation}
\omega^{\rm NO}(\cdot,T)=\bar\omega_0
+s_m(\varepsilon)\,m^{\rm NO}
+s_q(\varepsilon)\,q^{\rm NO} ,
\label{modelA_head}
\end{equation}
where $\bar\omega_0$ is computed exactly from the input $\omega_0$ and the
superscript NO marks an output of the neural operator. The network has two output
channels. The first, with its $x_1$-average removed, is $q^{\rm NO}$, and the
$x_1$-average of the second is $m^{\rm NO}(x_2)$. The remaining parts of the two
channels do not enter \eqref{modelA_head}. Thus $s_m\,m^{\rm NO}$ and
$s_q\,q^{\rm NO}$ approximate the two parts of $\omega(\cdot,T)-\bar\omega_0$ in
the mean/oscillation decomposition of Sec.~\ref{sec_meanosc}, $\delta\bar\omega$
and $\widetilde\omega(\cdot,T)$, respectively.

For each training amplitude $\varepsilon_j$, the scales $s_q(\varepsilon_j)$ and
$s_m(\varepsilon_j)$ are the root-mean-square values of $\widetilde\omega(\cdot,T)$
and $\delta\bar\omega$ over the grid points and the $n_j=1000$ training samples at
that amplitude,
\begin{equation}
s_q(\varepsilon_j)=\Big(\frac1{n_j}\sum_{i=1}^{n_j}\frac{\|\widetilde\omega_i(\cdot,T)\|_{L^2}^2}{4\pi^2}\Big)^{1/2},
\qquad
s_m(\varepsilon_j)=\Big(\frac1{n_j}\sum_{i=1}^{n_j}\frac{\|\delta\bar\omega_i\|_{L^2}^2}{4\pi^2}\Big)^{1/2}.
\label{scales_def}
\end{equation}
Between the training amplitudes, $\log s_q$ and $\log s_m$ are interpolated
linearly in $\log\varepsilon$. No power of $\varepsilon$ is imposed, but over the
training amplitudes $s_m$ grows as $\varepsilon^{2.00}$, with $s_m/\varepsilon^2$
between $5.7\times10^{-6}$ and $6.6\times10^{-6}$, and $s_q$ as $\varepsilon^{1.00}$.
Because the scales carry the growth of the targets with $\varepsilon$, the
network outputs $m^{\rm NO}$ and $q^{\rm NO}$ are of order one at every amplitude
and depend mainly on the normalized data $\omega_0/\varepsilon$ and
$\theta_0/\varepsilon$.

\paragraph{Training process.}
The training set is the $10^4$ samples of Sec.~\ref{sec_trainingdata}. An epoch is
one pass through the training set, and a batch is the set of $B$ samples used in one
gradient update. The
loss is the sum of the relative $L^2$ errors of the two outputs, each per sample and
averaged over the batch of size $B$,
\begin{equation}
\mathcal L=\frac1B\sum_{i=1}^{B}
\frac{\|s_q\,q^{\rm NO}_i-\widetilde\omega_i(\cdot,T)\|_{L^2}}{\|\widetilde\omega_i(\cdot,T)\|_{L^2}}
+\gamma\,\frac1B\sum_{i=1}^{B}
\frac{\|s_m\,m^{\rm NO}_i-\delta\bar\omega_i\|_{L^2}}{\|\delta\bar\omega_i\|_{L^2}},
\qquad \gamma=10.
\label{modelA_loss}
\end{equation}
The weight $\gamma$
emphasizes $\delta\bar\omega$, the part on which Theorem~\ref{thm:main}(b)
bears. Two preliminary runs on $900$ samples, with $\gamma=10$ and $\gamma=100$, showed that
the larger weight did not improve $\delta\bar\omega$ and left the oscillation output unlearned, its relative error staying at $1$, so
$\gamma=10$ is used. Optimization uses AdamW \cite{Loshchilov2019} with learning rate $3\times10^{-4}$,
weight decay $10^{-4}$ and batch size $B=8$, the learning rate being halved every
$200$ epochs. The network is trained for $600$ epochs, $7.5\times10^{5}$ gradient
updates of $3.47\times10^{7}$ parameters (complex weights counted once), on a
single NVIDIA Tesla V100S ($32$\,GB) GPU, in $10$\,h\,$59$\,min of wall-clock
time, $66$\,s per epoch.

\subsection{Surrogate accuracy and cost}
\label{sec_no_accuracy}

\paragraph{Accuracy.}
On the $2\times10^3$ test samples the surrogate \eqref{modelA_head} reproduces the
final vorticity $\omega(\cdot,T)$ to a median relative $L^2$ error between
$3.8\times10^{-5}$ and $2.4\times10^{-4}$ at the five interpolation amplitudes,
and between $4.0\times10^{-5}$ and $3.3\times10^{-4}$ at the ten training
amplitudes, the error growing with the amplitude.  Figure~\ref{fig:no_convergence} shows the training history and
Fig.~\ref{fig:no_prediction} three test samples, the best, the median and the
worst.

The field is accurate because most of it is $\bar\omega_0$, which the surrogate
adds exactly. The two learned parts are reproduced less accurately. The median
relative $L^2$ error is $4.0$ to $4.7$ percent for the oscillation
$\widetilde\omega(\cdot,T)$ and $23$ to $27$ percent for $\delta\bar\omega$, nearly
the same at every amplitude. Both parts are small compared with
$\omega(\cdot,T)$, so these errors barely affect the field.

The error of $\delta\bar\omega$, $s_m\,m^{\rm NO}-\delta\bar\omega$, lies at fine
vertical scales. Of its squared $L^2$ norm, $95$ to $99$ percent lies at
$|k_2|>16$, against $22$ to $26$ percent for $\delta\bar\omega$ itself, so the
network reproduces the large-scale profile of $\delta\bar\omega$ and misses its
fine structure. On the training samples this error is far smaller, below
$10^{-2}$ from epoch $201$ on and $2.4\times10^{-3}$ at the end, so the fine
structure is learned for the training profiles but does not carry over to new
ones. The $H^{-2}$ norm of Theorem~\ref{thm:main}(b) weights high wavenumbers
weakly. In this norm the median relative error of $\delta\bar\omega$ is $6$ to $10$
percent at every amplitude, and $\|\delta\bar\omega\|_{H^{-2}}$, which
Theorem~\ref{thm:main}(b) bounds by $C\varepsilon^2$ at every time, is reproduced
to a median relative error of $2.2$ percent (Sec.~\ref{sec_thm2_test}).

\begin{figure}[htbp]
\centering
\includegraphics[width=0.48\textwidth]{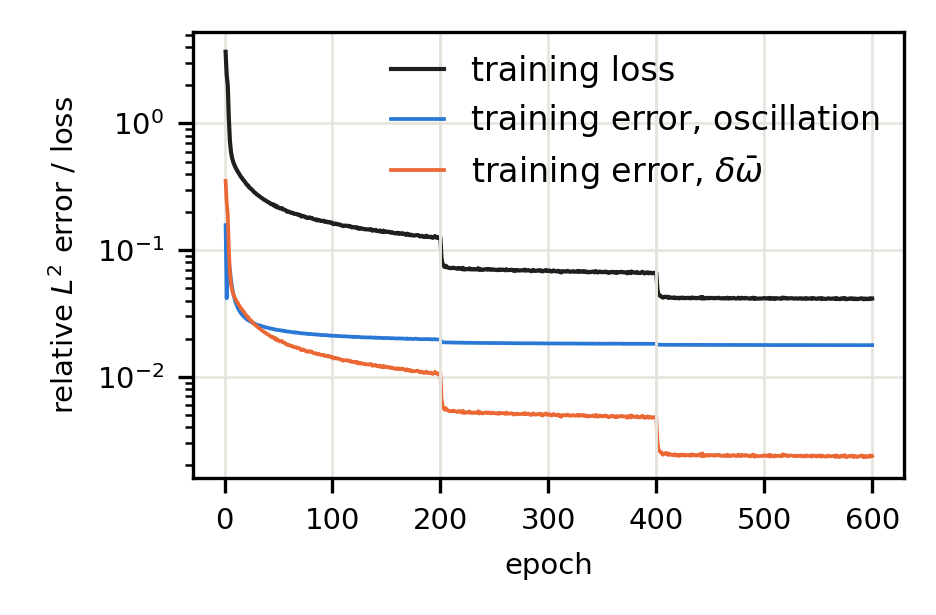}
\caption{Training history of the Fourier neural operator, versus epoch. Each
curve is the average over the batches of each epoch. Black:
the loss \eqref{modelA_loss}. Blue: the first term in \eqref{modelA_loss}, the
relative $L^2$ error of the oscillation on the training set. Orange: the second term
in \eqref{modelA_loss} without $\gamma$, the relative $L^2$ error of
$\delta\bar\omega$ on the training set. The learning rate is halved at epochs
$200$ and $400$.}
\label{fig:no_convergence}
\end{figure}

\begin{figure}[htbp]
\centering
\includegraphics[width=\textwidth]{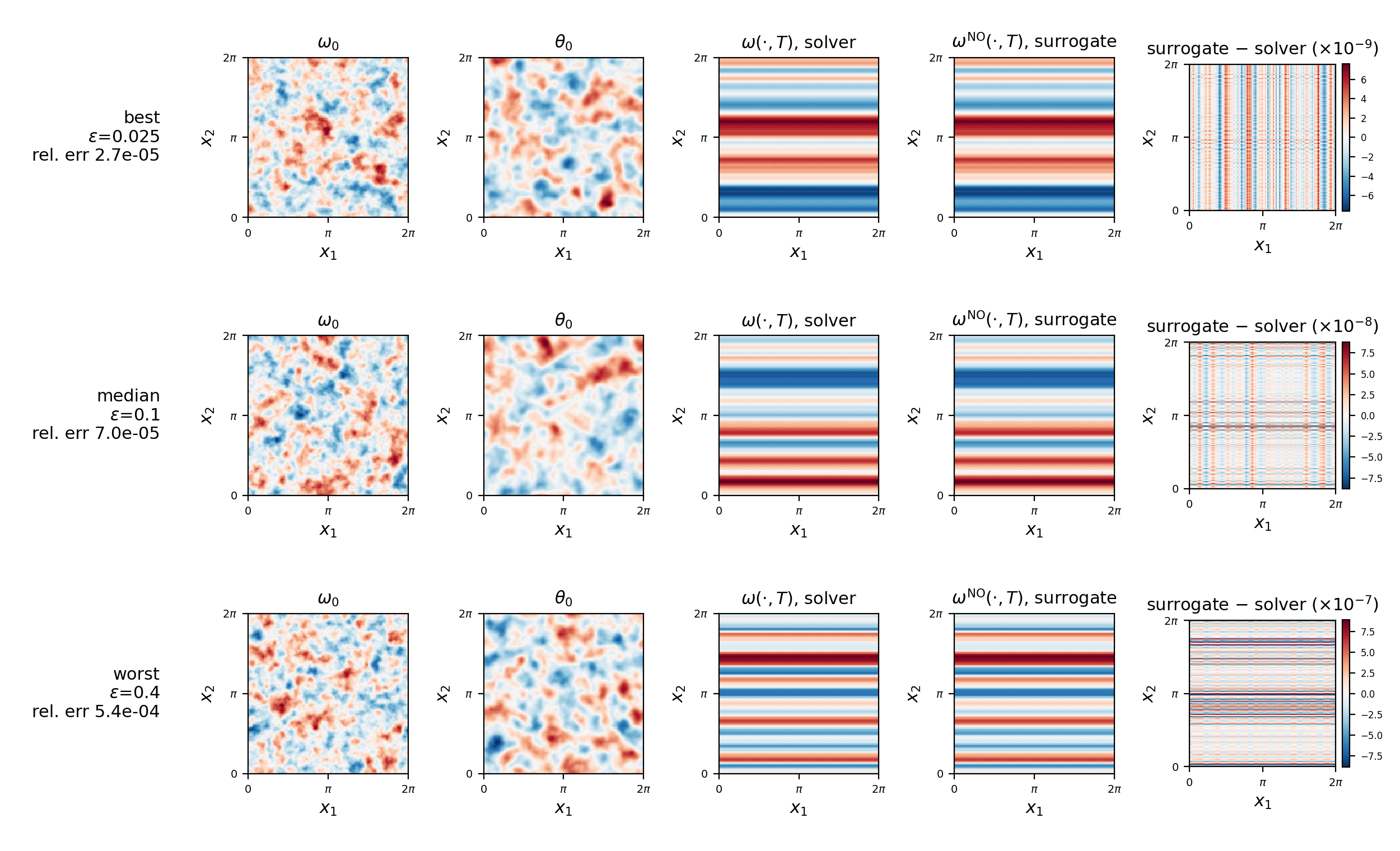}
\caption{The surrogate on the three samples of the interpolation set with the
smallest, median and largest relative $L^2$ error of $\omega(\cdot,T)$ among the
$10^{3}$ samples of that set ($2.7\times10^{-5}$, $7.0\times10^{-5}$ and $5.4\times10^{-4}$). Columns: the two input fields
$\omega_0$ and $\theta_0$, the spectral-solver final vorticity $\omega(\cdot,T)$,
the prediction $\omega^{\rm NO}(\cdot,T)$ of \eqref{modelA_head}, and the
pointwise error on its own scale. Fields are
drawn with $x_1$ horizontal. The final vorticity is banded horizontally, that is,
nearly independent of $x_1$. In the best case the error varies with $x_1$ and lies
in the oscillation. In the worst case it is banded horizontally and lies in the
fine structure of the stratified profile.}
\label{fig:no_prediction}
\end{figure}

\paragraph{Inference cost.}
Once trained, the surrogate is far cheaper to evaluate than the spectral solver
that generated its data, which costs $10.2$\,s per simulation on average
(Sec.~\ref{sec_trainingdata}).
On one NVIDIA Tesla V100S ($32$\,GB), a batch of $16$ initial fields is mapped to
the final vorticity \eqref{modelA_head}, including the computation of the inputs
and of \eqref{modelA_head}, in
$21.3$\,ms, that is, $1.33$\,ms per field, and a batch of $64$ in $73.6$\,ms,
$1.15$\,ms per field. Against the $10.2$\,s of the solver, at the batch of $16$, the surrogate is about $8\times10^{3}$ times faster per field.

\subsection{The quantity of Theorem~\ref{thm:main}(b): solver and surrogate}
\label{sec_thm2_test}

Theorem~\ref{thm:main}(b) bounds $\|\bar\omega(\cdot,t)-\bar\omega_0\|_{H^{-2}}$ at
every time by $C\varepsilon^2$,
with $C$ independent of the data, in the norm \eqref{norm_def}. For the whole field
the bound is weaker. Writing $\omega-\bar\omega_0=(\bar\omega-\bar\omega_0)+\widetilde\omega$
and using $\|\widetilde\omega\|_{H^{-2}}\le\|\widetilde\omega\|_{L^2}$, $\|\widetilde\omega\|_{L^2}\le\|\omega\|_{L^2}$ and $\|\omega\|_{L^2}\le\|u\|_{H^2}$
(Lemma~\ref{lem:meanosc}(c)) together with \eqref{V0},
\begin{equation}
\|\omega(\cdot,t)-\bar\omega_0\|_{H^{-2}}\le C\varepsilon^2+C_*\varepsilon
\qquad\text{for every } t\ge0,
\label{full_bound}
\end{equation}
which is of order $\varepsilon$ because the oscillation $\widetilde\omega$ is. For a sample $i$ at the final
time $T=100$ we define
\begin{equation}
\bar R_i=\frac{\|\bar\omega_i(\cdot,T)-\bar\omega_{0,i}\|_{H^{-2}}}{\varepsilon_i^{2}},
\qquad
R_i=\frac{\|\omega_i(\cdot,T)-\bar\omega_{0,i}\|_{H^{-2}}}{\varepsilon_i},
\label{R_def}
\end{equation}
and by Theorem~\ref{thm:main}(b), $\bar R_i\le C$ for every sample, while by
\eqref{full_bound} $R_i\le C\varepsilon_i+C_*$.

Table~\ref{tab:thm2} is built as follows. For each of the $2\times10^3$ test
samples, the norms in \eqref{R_def} are evaluated with the discrete Fourier
coefficients on the $128\times128$ grid, once with the solver's $\omega_i(\cdot,T)$,
giving $\bar R_i$ and $R_i$, and once with the surrogate's prediction
\eqref{modelA_head} in its place, giving $\bar R^{\rm NO}_i$ and $R^{\rm NO}_i$. At
each amplitude the table lists the sample mean, the standard deviation and the
maximum of $\bar R$, the sample mean and the maximum of $\bar R^{\rm NO}$, the
median and the maximum of
$|\bar R^{\rm NO}_i-\bar R_i|/\bar R_i$, the sample mean of $R$ and the median of
$|R^{\rm NO}_i-R_i|/R_i$. Figure~\ref{fig:thm2} shows the same data for $\bar R$. Panel (a)
plots $\bar R_i$ of every solver run against its amplitude, together with the
sample mean and the maximum at each amplitude for the solver and for the
surrogate, and panel (b) plots $\bar R^{\rm NO}_i$ against $\bar R_i$ for every
sample.

\begin{table}[htbp]
\caption{The quantities of \eqref{R_def} at $T=100$ on the $2\times10^3$ test
samples, from the spectral solver and from the surrogate. $\bar R$ is bounded by
Theorem~\ref{thm:main}(b) at every time. The interpolation amplitudes (top) were
not used in training. The columns $|\bar R^{\rm NO}-\bar R|/\bar R$ give the median
and the maximum over the amplitude of the per-sample
relative difference, and the last two columns the sample mean of $R$ and the median
of its relative difference. In the column set, train.\ marks the training
amplitudes and interp.\ the interpolation amplitudes. Standard deviations use
$n-1$.}
\label{tab:thm2}
\centering
\scriptsize
\setlength{\tabcolsep}{2.5pt}
\begin{tabular}{llrccccccccc}
\toprule
& & & \multicolumn{3}{c}{solver, $\bar R\times10^{6}$}
& \multicolumn{2}{c}{surrogate, $\bar R^{\rm NO}\times10^{6}$}
& \multicolumn{2}{c}{$|\bar R^{\rm NO}-\bar R|/\bar R$ (\%)}
& $R\times10^{7}$ & $\frac{|R^{\rm NO}-R|}{R}$ (\%)\\
\cmidrule(lr){4-6}\cmidrule(lr){7-8}\cmidrule(lr){9-10}\cmidrule(lr){11-11}\cmidrule(lr){12-12}
$\varepsilon$ & set & $n$ & mean & s.d. & max & mean & max & median & max & mean & median\\
\midrule
$0.025$ & interp. & $200$ & $1.43$ & $0.77$ & $4.24$ & $1.38$ & $4.05$ & $3.6$ & $8.0$ & $1.92$ & $0.7$\\
$0.050$ & interp. & $200$ & $1.44$ & $0.67$ & $3.82$ & $1.46$ & $3.78$ & $3.2$ & $8.2$ & $2.02$ & $0.6$\\
$0.100$ & interp. & $200$ & $1.41$ & $0.52$ & $3.13$ & $1.39$ & $3.06$ & $1.5$ & $5.3$ & $2.36$ & $0.7$\\
$0.200$ & interp. & $200$ & $1.46$ & $0.62$ & $4.08$ & $1.43$ & $3.86$ & $2.0$ & $8.2$ & $3.56$ & $1.1$\\
$0.400$ & interp. & $200$ & $1.43$ & $0.70$ & $4.22$ & $1.42$ & $3.98$ & $1.6$ & $7.3$ & $6.07$ & $1.3$\\
\midrule
$0.0210$ & train. & $100$ & $1.54$ & $0.56$ & $2.85$ & $1.53$ & $2.68$ & $1.9$ & $6.8$ & $1.85$ & $0.8$\\
$0.0297$ & train. & $100$ & $1.55$ & $0.67$ & $3.77$ & $1.45$ & $3.47$ & $6.6$ & $10.9$ & $2.00$ & $0.8$\\
$0.0420$ & train. & $100$ & $1.36$ & $0.65$ & $3.26$ & $1.43$ & $3.19$ & $6.0$ & $11.3$ & $1.93$ & $0.8$\\
$0.0595$ & train. & $100$ & $1.35$ & $0.66$ & $2.97$ & $1.34$ & $2.91$ & $1.7$ & $8.5$ & $2.10$ & $0.7$\\
$0.0841$ & train. & $100$ & $1.48$ & $0.59$ & $3.25$ & $1.44$ & $3.05$ & $1.8$ & $6.4$ & $2.30$ & $0.9$\\
$0.1189$ & train. & $100$ & $1.38$ & $0.55$ & $2.93$ & $1.40$ & $2.91$ & $2.0$ & $6.5$ & $2.63$ & $0.7$\\
$0.1682$ & train. & $100$ & $1.40$ & $0.56$ & $2.83$ & $1.39$ & $2.67$ & $1.3$ & $6.5$ & $3.08$ & $0.8$\\
$0.2378$ & train. & $100$ & $1.46$ & $0.52$ & $2.47$ & $1.43$ & $2.40$ & $2.0$ & $8.2$ & $4.02$ & $1.6$\\
$0.3364$ & train. & $100$ & $1.33$ & $0.57$ & $2.45$ & $1.29$ & $2.39$ & $3.2$ & $7.9$ & $4.92$ & $2.7$\\
$0.4757$ & train. & $100$ & $1.49$ & $0.56$ & $2.61$ & $1.51$ & $2.63$ & $1.4$ & $4.5$ & $7.39$ & $1.3$\\
\bottomrule
\end{tabular}
\end{table}

\begin{figure}[htbp]
\centering
\includegraphics[width=\textwidth]{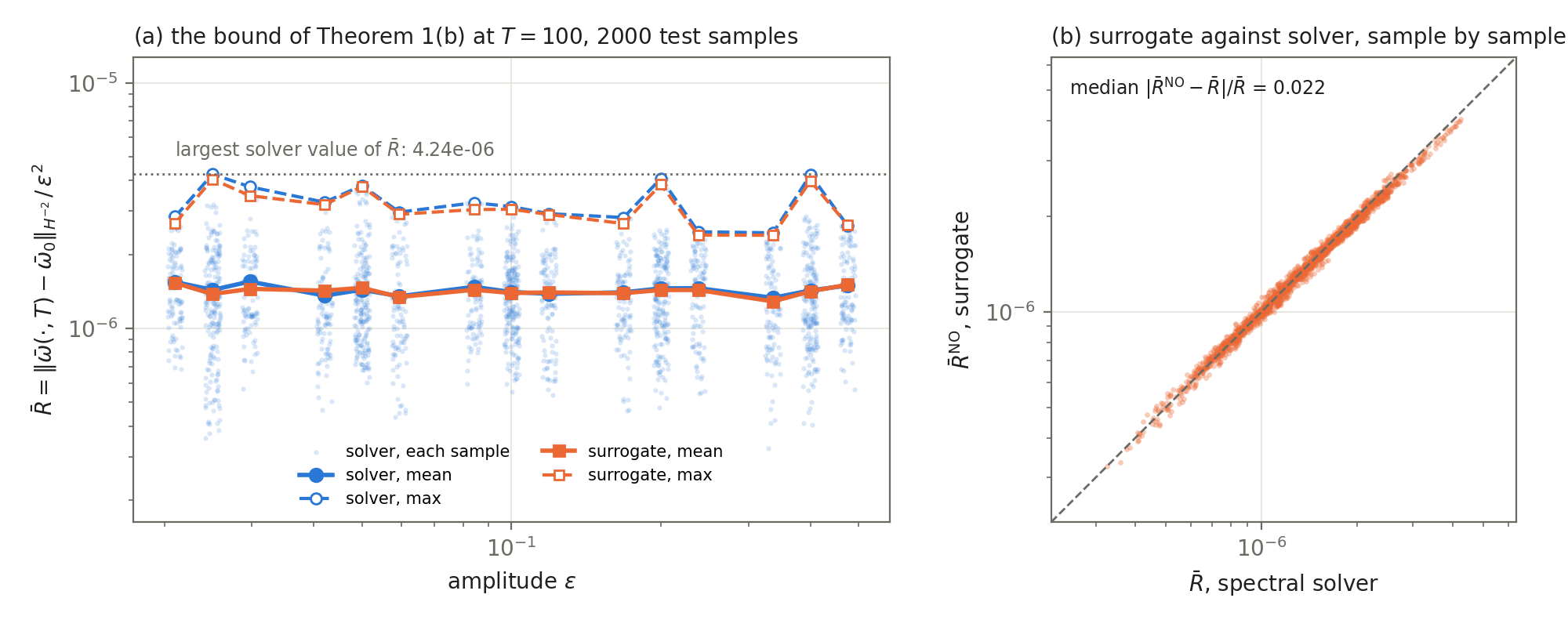}
\caption{(a) $\bar R=\|\bar\omega(\cdot,T)-\bar\omega_0\|_{H^{-2}}/\varepsilon^{2}$ at
$T=100$ for every test sample of the spectral solver (dots, spread slightly in
$\varepsilon$ for visibility), with the sample mean (solid line, filled markers)
and the maximum (dashed line, open markers) at each amplitude for the solver
(circles) and the surrogate (squares). The dotted line is the largest of the
$2\times10^3$ solver values. Both the sample mean and the maximum stay nearly the
same over a factor of $23$ in $\varepsilon$. (b) The surrogate's $\bar R^{\rm NO}$
against the solver's $\bar R$, sample by sample, with the line of equality. The
median printed in the panel is over all $2\times10^3$ samples.}
\label{fig:thm2}
\end{figure}

Table~\ref{tab:thm2} and Fig.~\ref{fig:thm2} show four things. First, $\bar R$
stays nearly the same as $\varepsilon$ varies. Its sample mean lies between
$1.33\times10^{-6}$ and $1.55\times10^{-6}$ at every amplitude from $0.021$ to
$0.476$, and a least-squares fit of the sample mean of
$\|\bar\omega(\cdot,T)-\bar\omega_0\|_{H^{-2}}$ to $c\,\varepsilon^{p}$ over all
$15$ amplitudes gives $p=1.99$, with $95$ percent bootstrap interval $[1.95,\,2.03]$. The maximum also stays nearly the same: it lies
between $2.45\times10^{-6}$ and $4.24\times10^{-6}$. All
$2\times10^3$ runs therefore satisfy
\[
\|\bar\omega(\cdot,T)-\bar\omega_0\|_{H^{-2}}\le4.25\times10^{-6}\,\varepsilon^{2},
\]
which is the form of Theorem~\ref{thm:main}(b) with an empirical constant, and over
the range tested nothing in the data suggests that the constant grows as
$\varepsilon$ decreases. All the samples come from one family of initial data,
Gaussian random fields of fixed covariance rescaled to the amplitude $\varepsilon$,
so the test shows that for this family the change is of second order with a
bounded ratio, and the
empirical constant depends on the covariance of the initial data.
Section~\ref{sec_search} shows that initial data of other shapes give a ratio
more than four times larger.

Second, at a fixed amplitude $\bar R$ varies from sample to sample. Its standard
deviation in Table~\ref{tab:thm2} is $0.36$ to $0.54$ times its sample mean. This
variation comes from the shapes of the initial data, $\omega_0/\varepsilon$ and
$\theta_0/\varepsilon$, which differ from sample to sample, and not from the
amplitude, since the relative spread is about the same at every $\varepsilon$.

Third, the surrogate's $\bar R^{\rm NO}_i$ is close to the solver's $\bar R_i$
in each individual run. At every amplitude the sample mean of $\bar R^{\rm NO}$
agrees with that of $\bar R$ to within $7$ percent, and its maximum and standard
deviation to within $8$ percent. Over all $2\times10^3$ test runs, the median of
$|\bar R^{\rm NO}_i-\bar R_i|/\bar R_i$ is $2.2$ percent, $84$ percent of the runs
have it below $5$ percent, and every run below $12$ percent. The largest value of
$\bar R^{\rm NO}$, $4.05\times10^{-6}$, is $5$ percent below the largest value of
$\bar R$. At every amplitude the median of $|\bar R^{\rm NO}_i-\bar R_i|/\bar R_i$
lies between $1.3$ and $6.6$ percent.

Fourth, $R$ behaves differently. Its sample mean is nearly the same, between
$1.8\times10^{-7}$ and $2.1\times10^{-7}$, for $\varepsilon\le0.06$, and rises to
$7.4\times10^{-7}$ at $\varepsilon=0.476$. To see why, write
\[
\omega(\cdot,T)-\bar\omega_0=\delta\bar\omega+\widetilde\omega(\cdot,T),
\]
where $\|\delta\bar\omega\|_{H^{-2}}=\varepsilon^2\bar R$ is of order
$\varepsilon^2$. Theorem~\ref{thm:main}(b) drives the oscillation $\widetilde\omega$
to zero, but at $T=100$ it is still present, because it is tied to the slowly
decaying temperature modes with $k_2=0$ of Sec.~\ref{sec_blade}. In every run at
least $99.99$ percent of $\|\widetilde\omega(\cdot,T)\|_{H^{-2}}^2$ lies on the
modes with $k_2=0$, and on these modes the equation
$\hat\omega(k_1,0,T)=i\hat\theta(k_1,0,T)/(\nu k_1)$ holds to within $0.15$ percent in
relative $H^{-2}$ norm, in the sense that
$\|Q_0\omega(\cdot,T)+\nu^{-1}\partial_1^{-1}Q_0\theta(\cdot,T)\|_{H^{-2}}$ is at most
$0.0015\,\|Q_0\omega(\cdot,T)\|_{H^{-2}}$, with $Q_0$ the vertical average \eqref{vert_avg} and $\partial_1^{-1}$ acting on the
modes with $k_1\neq0$. This equation is the balance $\nu\partial_{11}\omega+\partial_1\theta=0$
of \eqref{vort_omega}. The $H^{-2}$ norm of $\widetilde\omega(\cdot,T)$ lies between
$1.8\times10^{-7}\varepsilon$ and $2.0\times10^{-7}\varepsilon$ at every amplitude,
so the oscillation is of first order and keeps $R$ nearly constant for small
$\varepsilon$. At larger $\varepsilon$ the term $\delta\bar\omega$, of order
$\varepsilon^2$, becomes comparable, and $R$ rises. The surrogate's $R^{\rm NO}$,
which contains its prediction of the oscillation, matches $R$ to a median of $0.9$
percent and in every sample to within $7$ percent.

\subsection{Comparison with Fourier neural operators trained without the analysis}
\label{sec_baseline}

To isolate the effect of the design, two further networks were trained that
differ from the surrogate only in what the analysis decides. They use the same
$10^4$ training runs, the same Fourier neural operator with four layers, $64$
channels and the same retained modes, and the same optimizer, batch size,
learning-rate schedule, number of epochs and seed. Each predicts
$\omega(\cdot,T)$ whole, with one output channel and the loss
\begin{equation}
\mathcal L_{\rm base}=\frac1B\sum_{i=1}^{B}
\frac{\|\omega^{\rm NO}_i(\cdot,T)-\omega_i(\cdot,T)\|_{L^2}}{\|\omega_i(\cdot,T)\|_{L^2}},
\label{base_loss}
\end{equation}
the relative $L^2$ error commonly used for Fourier neural operators
\cite{Li2021FNO}. Neither network adds $\bar\omega_0$ exactly, splits the output
into $\delta\bar\omega$ and the oscillation, or scales the two parts. In the first
network, called plain, the inputs are $\omega_0$ and $\theta_0$, standardized per
channel, and the output is $\omega(\cdot,T)$, standardized by its training mean
and standard deviation, so the amplitude enters only through the size of the
fields. The second network, called scaled, receives the same inputs as the
surrogate, $\omega_0/\varepsilon$, $\theta_0/\varepsilon$ and
$\log(\varepsilon/0.1)$, and predicts $\omega(\cdot,T)/\varepsilon$. It isolates the
effect of the decomposition and of the separate scaling of the two parts.

\begin{table}[htbp]
\caption{The surrogate and two Fourier neural operators trained on the same data
with the loss \eqref{base_loss}, on the $10^3$ runs of each test set. The growth
exponent is the least-squares slope of the logarithm of the mean of
$\|\delta\bar\omega\|_{H^{-2}}$ over each amplitude against $\log\varepsilon$. All
values were computed on an NVIDIA V100S, the GPU used for training.}
\label{tab:baseline}
\centering
\footnotesize
\setlength{\tabcolsep}{5pt}
\begin{tabular}{lccc}
\toprule
& surrogate & plain & scaled\\
\midrule
\multicolumn{4}{l}{\emph{interpolation set}}\\
median $|\bar R^{\rm NO}-\bar R|/\bar R$ & $2.2\%$ & $3.6\times10^{3}\%$ & $2.9\times10^{3}\%$\\
largest $|\bar R^{\rm NO}-\bar R|/\bar R$ & $8.2\%$ & $8.1\times10^{4}\%$ & $9.1\times10^{4}\%$\\
median relative $L^2$ error of $\omega(\cdot,T)$ & $7.0\times10^{-5}$ & $2.8\times10^{-2}$ & $2.8\times10^{-2}$\\
growth exponent of $\|\delta\bar\omega\|_{H^{-2}}$ (solver $2.00$) & $2.00$ & $1.17$ & $0.95$\\
\midrule
\multicolumn{4}{l}{\emph{training-amplitude set}}\\
median $|\bar R^{\rm NO}-\bar R|/\bar R$ & $2.2\%$ & $3.6\times10^{3}\%$ & $3.0\times10^{3}\%$\\
largest $|\bar R^{\rm NO}-\bar R|/\bar R$ & $11\%$ & $4.1\times10^{4}\%$ & $5.0\times10^{4}\%$\\
median relative $L^2$ error of $\omega(\cdot,T)$ & $7.4\times10^{-5}$ & $3.0\times10^{-2}$ & $3.0\times10^{-2}$\\
growth exponent of $\|\delta\bar\omega\|_{H^{-2}}$ (solver $1.98$) & $1.99$ & $1.18$ & $0.98$\\
\bottomrule
\end{tabular}
\end{table}

Table~\ref{tab:baseline} shows the result. Both networks fit the training
fields closely, to a relative $L^2$ error below $10^{-3}$, and reproduce the
field of the test runs to about $3$ percent. Their median error in
$\delta\bar\omega$, however, is $30$ to $37$ times the size of $\delta\bar\omega$
itself, so the change of the average that they predict is determined by their
error in the field and not by the flow. The values of $\bar R$ they produce are,
on average over each amplitude, $6$ to $180$ times too large, most at the
smallest amplitudes, and the predicted change of the average grows as
$\varepsilon^{0.95}$ to $\varepsilon^{0.98}$ for the scaled network and
$\varepsilon^{1.17}$ to $\varepsilon^{1.18}$ for the plain one, against the second-order growth of Theorem~\ref{thm:main}(b), which the
solver confirms with the exponent $2.00$ and the surrogate reproduces. The field
itself is also predicted far less accurately, because the surrogate adds
$\bar\omega_0$, which makes up most of $\omega(\cdot,T)$, exactly.

\subsection{Surrogate investigation of a lower bound for the constant of Theorem~\ref{thm:main}(b)}
\label{sec_search}

Theorem~\ref{thm:main}(b) bounds $\bar R$ by a constant $C$ that depends only on
$\nu$, but gives no value for $C$, and the $2\times10^3$ solver runs of
Sec.~\ref{sec_thm2_test} only show that $C\ge4.24\times10^{-6}$. The surrogate
evaluates $\bar R^{\rm NO}$ in about $1$\,ms per sample on one GPU, against
$10.2$\,s per simulation for the solver. It is also differentiable: the gradient of
$\bar R^{\rm NO}$ with respect to the initial fields $(\omega_0,\theta_0)$ is
obtained by automatic differentiation at the cost of about one further
evaluation, whereas the solver would need an adjoint computation backward in time
from $T$ to $0$. We use the speed to screen many initial fields and the gradient to
start an ascent from the screened fields with the largest $\bar R^{\rm NO}$, in
search of initial fields with a larger $\bar R$. The search is carried out separately at each of the five interpolation
amplitudes, and the amplitude is held fixed. Only the shape of the fields
changes, and their norms stay $\|u_0\|_{H^2}=\|\theta_0\|_{H^2}=\varepsilon/2$.
Since $\bar R$ is divided by $\varepsilon^2$ and Theorem~\ref{thm:main}(b) bounds it
by the same constant at every amplitude, the values found at the different
amplitudes are compared with one another. Every value of $\bar R$ reported in this subsection is recomputed with
the spectral solver.

First, the surrogate evaluated $\bar R^{\rm NO}$ for $10^6$ new samples of the
family of Sec.~\ref{sec_trainingdata} at each of the five interpolation
amplitudes, $5\times10^6$ samples in all, in $83$ minutes of GPU time. The
solver would need about $1.6$ years for the same task. At each amplitude, the
average of $\bar R^{\rm NO}$ over the $10^6$ samples lies between
$1.36\times10^{-6}$ and $1.44\times10^{-6}$, close to the averages of the solver's
$\bar R$ in Table~\ref{tab:thm2}.
However, at each amplitude a few hundred of the $10^6$ samples have
$\bar R^{\rm NO}>4.25\times10^{-6}$, a value that no test sample reaches. The
largest value of $\bar R^{\rm NO}$ over all $5\times10^6$ samples is
$7.34\times10^{-6}$. The solver,
run on the $16$ samples with the largest $\bar R^{\rm NO}$ at each amplitude,
gives $\bar R>4.25\times10^{-6}$ for all $80$ of them. The largest value,
$\bar R=8.32\times10^{-6}$, comes from the same sample that gives the largest
$\bar R^{\rm NO}$.

Second, starting from the $8$ samples with the largest $\bar R^{\rm NO}$ at each
amplitude, $40$ starting samples in all, each kept at its own amplitude, we maximized $\log\bar R^{\rm NO}$ by gradient ascent with the Adam
method, step size $0.02$ and $2000$ steps. The variables are the complex Gaussian
coefficients $\xi^\omega_k$ and $\xi^\theta_k$, whose real and imaginary parts are
independent standard normal variables at the start, that define the initial data
through
\[
\hat\omega_0(k)=c_\omega\,(\tau^2+|k|^2)^{-\alpha_\omega/2}\,\xi^\omega_k,
\qquad
\hat\theta_0(k)=c_\theta\,(\tau^2+|k|^2)^{-\alpha_\theta/2}\,\xi^\theta_k,
\]
with $\alpha_\omega=2.5$, $\alpha_\theta=3.5$ and $\tau=7$ as in Sec.~\ref{sec_trainingdata}, where the
constants $c_\omega,c_\theta>0$ are recomputed after every step so that
$\|u_0\|_{H^2}=\|\theta_0\|_{H^2}=\varepsilon/2$. After every step the
Gaussian coefficients are also rescaled so that, for each $s\ge1$, the sums of
$|\hat\omega_0(k)|^2$ and of $|\hat\theta_0(k)|^2$ over the annulus
$s-\tfrac12<|k|\le s+\tfrac12$ keep their initial values, up to the common
factors $c_\omega^2$ and $c_\theta^2$. Only the phases of the Fourier coefficients
and the distribution of their moduli among the wavenumbers of one annulus
change. This keeps the initial data in the class
on which the surrogate was trained, which is meant to keep its predictions
reliable along the ascent, as in model-based optimization restricted to the region where the
model is trusted \cite{Trabucco2021}. Fields with a prescribed power spectrum and
modified phases are used in the same way to separate the role of the phases from
that of the spectrum \cite{Theiler1992}.

The ascent gives $40$ paths in the space of initial data, one from each starting
sample. We ran the solver on the $16$ paths with the largest $\bar R^{\rm NO}$ at
step $2000$, at steps $0$, $100$, $300$, $500$, $1000$, $1500$ and $2000$ of each
path (Fig.~\ref{fig:search}(a)). Up to step $300$ the surrogate
agrees with the solver to a median relative difference of about $10$ percent, and
the solver's $\bar R$ first grows along every path, to a largest value
$\bar R=1.87\times10^{-5}$ at step $500$, $4.4$ times the largest value of the
test set. Recomputed on a $256\times256$ grid, or with half the time
step, this largest value changes by less than $10^{-5}$ in relative terms. Further on, the surrogate overestimates $\bar R$, by a median of $50$
percent at step $2000$, the ascent follows its errors, and the solver's $\bar R$
decreases. Figure~\ref{fig:search}(b)--(e) shows the case with the largest
difference, $\bar R^{\rm NO}=1.31\times10^{-5}$ against $\bar R=5.9\times10^{-6}$.

Since Theorem~\ref{thm:main}(b) bounds $\bar R$ at every time, the largest value
of $\bar R$ found over all the amplitudes tested, $\varepsilon$ from $0.025$ to
$0.4$, gives $C\ge1.87\times10^{-5}$ for $\nu=\eta=1$, provided these amplitudes
lie below the threshold $\varepsilon_*$.

\begin{figure}[htbp]
\centering
\includegraphics[width=\textwidth]{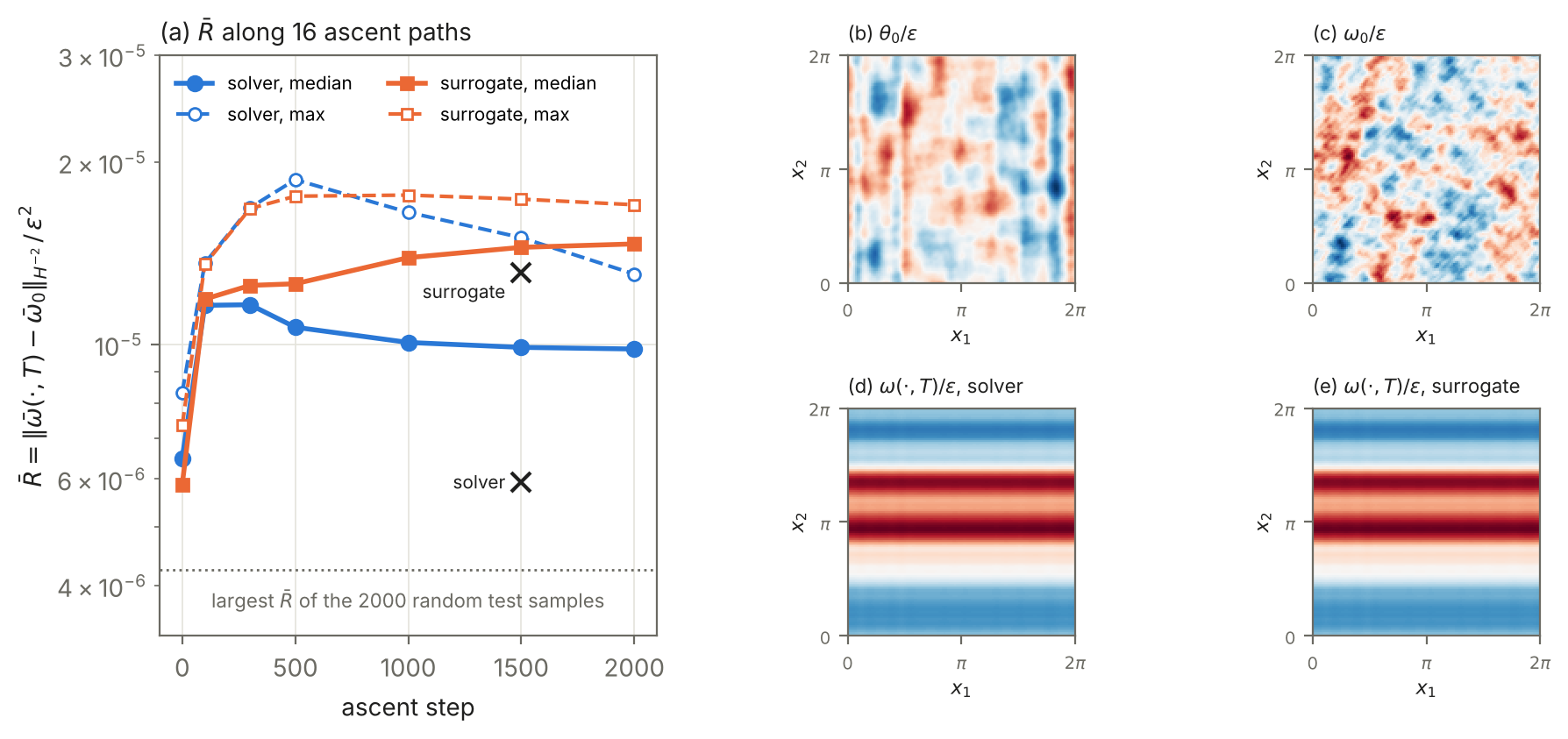}
\caption{Gradient ascent of the surrogate's $\bar R^{\rm NO}$ with the sum of
$|\hat\omega_0(k)|^2$ and of $|\hat\theta_0(k)|^2$ over each wavenumber annulus
held fixed. (a) $\bar R$ at seven steps of $16$ ascent paths, from the spectral
solver (circles) and the surrogate (squares), median (solid line, filled markers)
and maximum (dashed line, open markers) over the paths. The crosses mark the case
of (b)--(e), the largest relative difference between the surrogate and the
solver. The dotted line is the largest $\bar R$ of the $2\times10^3$ test samples.
(b) Initial temperature $\theta_0/\varepsilon$ and (c) initial vorticity
$\omega_0/\varepsilon$ of that case, at step $1500$ of its path and
$\varepsilon=0.1$. (d), (e) Its vorticity $\omega(\cdot,T)/\varepsilon$ at $T=100$
from the solver and from the surrogate, with one color scale.}
\label{fig:search}
\end{figure}

\section{Discussion and conclusions}
\label{sec_discussion}

In this paper we designed a
neural operator from a rigorous analysis of the long-time behavior of the
two-dimensional anisotropic Boussinesq system on the periodic box $\mathbb T^2$,
with horizontal viscosity $\nu$ and vertical thermal diffusivity $\eta$. The
operator maps a small perturbation of hydrostatic equilibrium to the vorticity at
a late time, and it captures the main features of the analysis. It reproduces,
sample by sample, the $H^{-2}$ norm of the change of the horizontally averaged
vorticity, the quantity that the analysis bounds, with a median error of $2.2$ percent and at most $12$ percent
over $2\times10^3$ test runs, together with its growth as $\varepsilon^2$, and
the oscillation of the vorticity about its horizontal average, of order
$\varepsilon$. A single network, trained on $10^4$ solver runs at ten amplitudes
spanning a factor of $23$, serves all amplitudes, including five not used in
training. It also gives a quantitative lower bound for
the constant in the estimate, which the analysis does not provide.

The long-time analysis guides the design of the network. First, the network does not learn the
final vorticity whole. Its prediction is the initial averaged vorticity
$\bar\omega_0$ plus a learned approximation of
the change $\omega(\cdot,T)-\bar\omega_0$, split into the change of the
horizontal average and the remaining oscillation. The loss measures the relative error of each part separately and gives the
change of the average ten times the weight of the oscillation. Second, the inputs
are the initial fields divided by $\varepsilon$, and each part of the output is
divided by its size at that amplitude, of order $\varepsilon^2$ for the change of
the average and $\varepsilon$ for the oscillation. The network outputs are therefore of order one at every
amplitude. Third,
since the temperature decays to zero, the network predicts only the vorticity at
the final time $T=100$.

Departing from the theory leads to failure. When the final temperature was
included in the training targets in a preliminary experiment, the training did not converge. When the same
network is trained with the standard loss, the error of the whole final vorticity
between the network and the solver, it learns the training fields well but
misses the quantity bounded by the analysis by a factor of about $30$ in the
median, and the change of the horizontal average that it predicts grows only about
linearly in $\varepsilon$, instead of with the second-order growth that the
analysis proves.

Because the network evaluates an initial datum in milliseconds, against
$10.2$\,s for the solver, and gives the gradient of its prediction with respect
to the initial data, it can search the space of initial data far more
extensively than the solver. Screening $5\times10^6$ random initial data took
$83$ minutes with the network and would take about $1.6$ years with the
solver. A gradient ascent of its prediction, with every candidate checked
with the solver, finds initial data whose change of the horizontal average is
$4.4$ times the largest among the $2\times10^3$ test runs, and gives the lower
bound $C\ge1.87\times10^{-5}$ for the constant in the estimate
$\|\bar\omega(\cdot,t)-\bar\omega_0\|_{H^{-2}}\le C\varepsilon^2$.

The approach extends to other equations for which an analysis identifies a
nearly conserved quantity and gives the order in the amplitude of each part of
the solution. Such knowledge lets a neural
operator learn a small but important quantity accurately, and the trained
operator in turn estimates quantities, such as the constants in the estimates,
that the analysis does not provide.

\section*{Acknowledgments}

J. Wu was partially supported by the National Science Foundation (NSF) under
grants DMS-2104682 and DMS-2309748. X. Zheng was partially supported by the NSF
under grant DMS-2309747. This work was supported in part by computational
resources and services provided by the High Performance Computing Center (HPCC) of the Institute for Cyber-Enabled
Research at Michigan State University.

\appendix

\section{The uniform stability estimate}
\label{app:stability}

We record here the global stability estimate on $\mathbb T^2$. It is Theorem~1.1 of
\cite{GaoWuZheng2026}, which uses the norms \eqref{norm_def} and the
smallness condition $(\|u_0\|_{H^2}^2+\|\theta_0\|_{H^2}^2)^{1/2}\le\varepsilon$,
implied by \eqref{data_hyp}.

\begin{thm}[Uniform $H^2$ stability on $\mathbb T^2$, \cite{GaoWuZheng2026}]
\label{thm:stability}
There is $\varepsilon_*=\varepsilon_*(\nu,\eta)>0$ with the following property.
Let $0<\varepsilon\le\varepsilon_*$ and let the initial data
$(u_0,\theta_0)\in H^2(\mathbb T^2)$ be divergence free and satisfy
\eqref{data_hyp}. Then \eqref{per001}--\eqref{per003} possesses a global
solution, unique in the class
\begin{equation}
\begin{gathered}
(u,\theta)\in L^\infty\big(0,\infty;H^2(\mathbb T^2)\big)
\cap\bigcap_{0\le s<2}C\big([0,\infty);H^s(\mathbb T^2)\big),\\
\partial_t(u,\theta)\in L^\infty\big(0,\infty;L^2(\mathbb T^2)\big),
\end{gathered}
\label{sol_class}
\end{equation}
and there are an absolute constant $C_*$ and a constant $C_\nu$ depending only on
$\nu$ such that
\begin{align}
&\sup_{t\ge0}\big(\|u(t)\|_{H^2}+\|\theta(t)\|_{H^2}\big)\le C_*\,\varepsilon,
\label{V0}\\
&\int_0^{\infty}\big(\nu\|\partial_1u\|_{H^2}^2
+\eta\|\partial_2\theta\|_{H^2}^2\big)\,d\tau\le C_\nu\,\varepsilon^2,
\label{V1}\\
&\int_0^{\infty}\|\partial_1Q_0\theta\|_{L^2}^2\,d\tau\le C_\nu\,\varepsilon^2 .
\label{V2}
\end{align}
One may take $C_*=8\sqrt2$ and $C_\nu=64(2+\nu)/\min\{1,\nu^2\}$. The thermal
diffusivity $\eta$ enters only through the smallness threshold $\varepsilon_*$.
\end{thm}

\begin{rmk}[The estimate used in the proof]
\label{rmk:unifbound}
Theorem~\ref{thm:stability} yields the following estimate, on which the proof of
Theorem~\ref{thm:main} in Appendix~\ref{app:thm} draws. In addition, the proof of
part~(b) uses \eqref{V1} directly, so that its constant depends on $\nu$ only.
There is $C_0=C_0(\nu,\eta)$ such that
\begin{multline}
\sup_{t\ge0}\Big(\|u(t)\|_{H^2}^2+\|\theta(t)\|_{H^2}^2\Big)
+2\nu\!\int_0^{\infty}\!\!\|\partial_1u\|_{H^2}^2\,d\tau
+2\eta\!\int_0^{\infty}\!\!\|\partial_2\theta\|_{H^2}^2\,d\tau\\
+\int_0^{\infty}\!\!\|\partial_1\theta\|_{L^2}^2\,d\tau
\le C_0\,\varepsilon^2 .
\label{unif_bound}
\end{multline}
The first three terms are \eqref{V0} and \eqref{V1}. Only the last one requires a proof. Decompose $\partial_1\theta
=\partial_1Q_0\theta+\partial_1\theta^{\perp}$, which is orthogonal in
$L^2(\mathbb T^2)$ by Lemma~\ref{lem:meanosc}(g). For every $x_1$ the function
$x_2\mapsto\partial_1\theta^{\perp}(x_1,x_2)$ has zero mean on $[0,2\pi]$, so the
vertical Poincar\'e inequality \eqref{poincare_y} applies and gives, using
$\partial_2Q_0\theta=0$,
\[
\|\partial_1\theta^{\perp}\|_{L^2}
\le\|\partial_2\partial_1\theta^{\perp}\|_{L^2}
=\|\partial_1\partial_2\theta\|_{L^2}
\le\|\partial_2\theta\|_{H^1}.
\]
Therefore
$\int_0^{\infty}\|\partial_1\theta^{\perp}\|_{L^2}^2\,d\tau
\le\eta^{-1}\int_0^{\infty}\eta\|\partial_2\theta\|_{H^2}^2\,d\tau
\le\eta^{-1}C_\nu\varepsilon^2$,
and adding \eqref{V2} bounds the last term of \eqref{unif_bound}. Collecting the
four contributions, one may take $C_0=C_*^2+(3+\eta^{-1})C_\nu$, which depends on
$\nu$ and $\eta$ only.
\end{rmk}

\section{Proof of Theorem~\ref{thm:main}}
\label{app:thm}

We first record Barbalat's lemma \cite{Barbalat1959,Khalil2002}, which turns the
finite dissipation integrals of \eqref{unif_bound} into convergence to zero.

\begin{lem}[Barbalat]
\label{lem:barbalat}
Let $g:[0,\infty)\to[0,\infty)$ satisfy $\int_0^\infty g(\tau)\,d\tau<\infty$ and
be uniformly continuous, for instance Lipschitz. Then $g(t)\to0$ as $t\to\infty$.
\end{lem}

\subsection{Proof of part (a)}
\label{app:thm_a}

\begin{proof}[Proof of Theorem~\ref{thm:main}(a)]
Global existence, uniqueness and the bound \eqref{unif_bound} are furnished by
Theorem~\ref{thm:stability} and Remark~\ref{rmk:unifbound} of
Appendix~\ref{app:stability}. We repeatedly use
that, by \eqref{unif_bound},
$\sup_{t\ge0}(\|u(t)\|_{H^2}+\|\theta(t)\|_{H^2})\le C\varepsilon$. Throughout
the proof of part~(a), $C$ denotes a constant that may depend on the viscosity
$\nu$ and the diffusivity $\eta$ but on nothing else, and that may change from one
occurrence to the next. In the proof of part~(b) the dependence of the constant
is tracked explicitly.

Write $v$ for either $u$ or $\theta$, and let $D$ denote a first-order spatial
derivative, so that $D\in\{\partial_1,\partial_2\}$. By \eqref{sol_class} the
solution supplied by Theorem~\ref{thm:stability} has $\partial_tv$ only in
$L^2$, hence $\partial_tDv=D\partial_tv$ is only known to lie in $H^{-1}$, and
the chain rule $\frac{d}{dt}\|Dv\|_{L^2}^2=2(\partial_tDv,Dv)_{L^2}$ has no
meaning as an $L^2$ inner product. It is supplied instead by the following standard fact
\cite[Ch.~III, Lemma~1.2]{Temam1984}, in which the pairing is the duality
between $H^{-1}$ and $H^1$. If
$w\in L^2_{\rm loc}([0,\infty);H^1(\mathbb T^2))$ and
$\partial_tw\in L^2_{\rm loc}([0,\infty);H^{-1}(\mathbb T^2))$, then
$t\mapsto\|w(t)\|_{L^2}^2$ is absolutely continuous and
\begin{equation}
\frac{d}{dt}\|w\|_{L^2}^2
=2\big\langle\partial_tw,\,w\big\rangle_{H^{-1},H^1}
\qquad\text{for almost every }t .
\label{ac_identity}
\end{equation}
Every function $g$ to which Lemma~\ref{lem:barbalat} is applied below is of the
form $g=\|w\|_{L^2}^2$ with $w=Dv$. By \eqref{sol_class}, $w\in L^\infty(0,\infty;H^1)$ and
$\partial_tw=D\partial_tv\in L^\infty(0,\infty;H^{-1})$, so \eqref{ac_identity}
applies and, moving $D$ onto the second factor,
$|dg/dt|=2\big|\int_{\mathbb T^2}\partial_tv\cdot D^2v\,dx\big|
\le2\|\partial_tv\|_{L^2}\|D^2v\|_{L^2}$, which is bounded by \eqref{sol_class}
and \eqref{unif_bound}. Hence each such $g$ is Lipschitz on $[0,\infty)$. 

\medskip
\textbf{Step 1: Preservation of the zero means.}
Write $m_\theta=\int_{\mathbb T^2}\theta$ and $m_u=\int_{\mathbb T^2}u_2$. Both are
Lipschitz on $[0,\infty)$, because $\partial_t(u,\theta)\in L^\infty(0,\infty;L^2)$
by \eqref{sol_class} and $\mathbb T^2$ has finite measure, and by the identities of
Sec.~\ref{sec_hyp} they satisfy
$m_\theta'=-m_u$ and $m_u'=m_\theta$ for almost every $t$. Hence
$m_\theta^2+m_u^2$ is absolutely continuous with
$(m_\theta^2+m_u^2)'=-2m_\theta m_u+2m_um_\theta=0$ almost everywhere, so it is
constant. It vanishes at $t=0$, and therefore $m_\theta\equiv m_u\equiv0$. Since $u$ is divergence-free and
$\int_{\mathbb T^2}u_2=0$, Lemma~\ref{lem:meanosc}(d) gives $\bar u_2\equiv0$ and $u_2=\widetilde u_2$.

\medskip
\textbf{Step 2: Decay of the oscillatory part and of $u_2$.}
We first show $\|\widetilde u(t)\|_{H^1}\to0$. Set $g_1(t)=\|\partial_1u(t)\|_{L^2}^2$.
By \eqref{unif_bound}, $\int_0^\infty g_1\,d\tau\le\int_0^\infty\|\partial_1u\|_{H^2}^2\,d\tau<\infty$,
and $g_1$ is Lipschitz, as shown above. Lemma~\ref{lem:barbalat} gives $g_1(t)\to0$, that is, $\|\partial_1u(t)\|_{L^2}\to0$.
By the horizontal Poincar\'e inequality \eqref{poincare_x}
(Lemma~\ref{lem:meanosc}(e)),
$\|\widetilde u\|_{L^2}\le\|\partial_1u\|_{L^2}\to0$, and interpolation with the
uniform bound yields
\[
\|\widetilde u(t)\|_{H^1}\le C\|\widetilde u\|_{L^2}^{1/2}\|\widetilde u\|_{H^2}^{1/2}
\le C\varepsilon^{1/2}\|\widetilde u\|_{L^2}^{1/2}\to0 .
\]
Since $u_2=\widetilde u_2$ by Step~1,
$\|u_2(t)\|_{H^1}\le\|\widetilde u(t)\|_{H^1}\to0$.

\medskip
\textbf{Step 3: Decay of the temperature.}
Applying Lemma~\ref{lem:barbalat} to the Lipschitz functions $g_2(t)=\|\partial_1\theta\|_{L^2}^2$
and $g_3(t)=\|\partial_2\theta\|_{L^2}^2$, whose time integrals are finite by
\eqref{unif_bound}, gives $\|\partial_1\theta(t)\|_{L^2}\to0$ and
$\|\partial_2\theta(t)\|_{L^2}\to0$. By Step~1, $\int_{\mathbb T^2}\theta(t)=0$, so the Fourier
coefficient of $\theta(t)$ at $k=0$ vanishes, and since $|k|\ge1$ for every other
wavenumber, the Poincar\'e inequality on $\mathbb T^2$ gives
\[
\|\theta(t)\|_{L^2}\le\|\nabla\theta(t)\|_{L^2}\to0 ,
\]
hence $\|\theta(t)\|_{H^1}\to0$.

\medskip
\textbf{Step 4: Decay in $H^s$, $s<2$.}
For $0\le s\le2$ and every $f\in H^2(\mathbb T^2)$, H\"older's inequality applied
to the Fourier sum \eqref{norm_def} with the weight
$(1+|k|^2)^{s}=\big((1+|k|^2)^{2}\big)^{s/2}$ gives
$\|f\|_{H^s}\le\|f\|_{L^2}^{1-s/2}\,\|f\|_{H^{2}}^{s/2}$. By \eqref{unif_bound} the
$H^2$ norms of $u_2(t)$ and $\theta(t)$ are bounded by $C\varepsilon$ for all
$t$, and by Steps~2 and~3 their $L^2$ norms tend to zero. Hence
$u_2(t)\to0$ and $\theta(t)\to0$ in $H^s$ for every $0\le s<2$, and for $s<0$
the claim follows from the case $s=0$.
\end{proof}

\subsection{Proof of part (b)}
\label{app:cor}

\begin{proof}[Proof of Theorem~\ref{thm:main}(b)]
\textbf{Step 1: The vorticity becomes stratified.} By Step~2 of the proof of
part~(a), $\|\widetilde u(t)\|_{H^1}\to0$. The oscillatory part of the
vorticity is $\widetilde\omega=\omega-\bar\omega=\partial_1\widetilde u_2-\partial_2\widetilde u_1$
(Lemma~\ref{lem:meanosc}(b)), so
$\|\omega(t)-\bar\omega(\cdot,t)\|_{L^2}=\|\widetilde\omega(t)\|_{L^2}\le C\|\widetilde u(t)\|_{H^1}\to0$.

\medskip
\textbf{Step 2: The horizontal average converges.} From \eqref{mean_eq}, for $0\le t_1\le t_2$,
\begin{multline*}
\|\bar\omega(\cdot,t_2)-\bar\omega(\cdot,t_1)\|_{H^{-2}}
\le\int_{t_1}^{t_2}\|\partial_{22}\overline{\widetilde u_1\widetilde u_2}\|_{H^{-2}}\,dt\\
\le\int_{t_1}^{t_2}\|\widetilde u_1\widetilde u_2\|_{L^2}\,dt
\le C\int_{t_1}^{t_2}\|\widetilde u(t)\|_{H^1}^2\,dt,
\end{multline*}
where we used $\|\partial_{22}g\|_{H^{-2}}\le\|g\|_{L^2}$, $\|\overline g\|_{L^2}\le\|g\|_{L^2}$
(Lemma~\ref{lem:meanosc}(c)),
and the Ladyzhenskaya inequality $\|\widetilde u_1\widetilde u_2\|_{L^2}\le C\|\widetilde u\|_{H^1}^2$.
By Lemma~\ref{lem:meanosc}(e), $\|\widetilde u\|_{H^1}\le\sqrt2\,\|\partial_1u\|_{H^1}$, and with \eqref{V1} this gives
\[
\int_0^\infty\|\widetilde u\|_{H^1}^2\,dt\le 2\int_0^\infty\|\partial_1u\|_{H^2}^2\,dt
\le2\nu^{-1}C_\nu\varepsilon^2 .
\]
Therefore the right-hand side of the preceding estimate tends to $0$ as
$t_1\to\infty$, and $\bar\omega(\cdot,t)$ is Cauchy in $H^{-2}$ and converges to
some $\bar\omega_\infty$. Integrating \eqref{mean_eq} from $0$ to $\infty$ gives
the representation \eqref{limit_rep}. The same estimate with $t_1=0$ and
$t_2=t$ gives $\|\bar\omega(\cdot,t)-\bar\omega_0\|_{H^{-2}}\le C\varepsilon^2$
for every $t\ge0$, and with $t_2=\infty$ the bound on
$\bar\omega_\infty-\bar\omega_0$. The
Ladyzhenskaya and Poincar\'e constants are absolute, so the constant depends only
on $\nu$.

\medskip
\textbf{Step 3: The whole field converges to the same limit.} By
Step~1, $\omega(\cdot,t)-\bar\omega(\cdot,t)\to0$ in
$L^2(\mathbb T^2)$, hence in $H^{-2}(\mathbb T^2)$, the embedding
$L^2\hookrightarrow H^{-2}$ being continuous. Adding the convergence of Step~2
gives $\omega(\cdot,t)\to\bar\omega_\infty$ in $H^{-2}$, and by continuity of the
norm $\|\omega(\cdot,t)-\bar\omega_0\|_{H^{-2}}\to\|\bar\omega_\infty-\bar\omega_0\|_{H^{-2}}$,
which Step~2 bounds by $C\varepsilon^2$.

\medskip
\textbf{Step 4: Regularity of the limit and convergence in $H^s$, $s<1$.} Since
$|\hat\omega_k|\le|k|\,|\hat u_k|$ and $(1+|k|^2)|k|^2\le(1+|k|^2)^2$,
\eqref{unif_bound} gives $\|\omega(t)\|_{H^1}\le C\varepsilon$ for all $t$. Each
Fourier coefficient of $\omega(t)$ converges to that of $\bar\omega_\infty$ by
Step~3, so Fatou's lemma gives $\|\bar\omega_\infty\|_{H^1}\le C\varepsilon$, and
$\bar\omega_\infty\in H^1(\mathbb T^2)$. For $-2\le s\le1$, H\"older's inequality
applied to the Fourier sum with
$(1+|k|^2)^s=\big((1+|k|^2)^{-2}\big)^{(1-s)/3}\big(1+|k|^2\big)^{(s+2)/3}$ gives
\[
\|f\|_{H^s}\le\|f\|_{H^{-2}}^{(1-s)/3}\,\|f\|_{H^1}^{(s+2)/3}.
\]
With $f=\omega(t)-\bar\omega_\infty$, which is bounded in $H^1$ by $C\varepsilon$
and tends to zero in $H^{-2}$ by Step~3, this gives
$\omega(\cdot,t)\to\bar\omega_\infty$ in $H^s$ for every $-2\le s<1$. For $s<-2$
the $H^s$ norm is bounded by the $H^{-2}$ norm, and the convergence follows from
the case $s=-2$.
\end{proof}

\section{Proof of Proposition~\ref{prop:linear}}
\label{app:linear}

\begin{proof}
Throughout, $k\neq0$. The Biot--Savart law
gives $\hat u_2=-ik_1\hat\omega/|k|^2$. Multiplying the
Fourier transform of the first equation of \eqref{lin_sys} by $-ik_1/|k|^2$ gives
the first equation of \eqref{lin_mode}, and the second follows directly.

(a) Let $k_1\neq0$, $k_2\neq0$ and
$E_k=|\hat u_2|^2+(k_1^2/|k|^2)|\hat\theta|^2$. By \eqref{lin_mode},
\[
\frac{dE_k}{dt}=-2\nu k_1^2|\hat u_2|^2-2\eta k_2^2\frac{k_1^2}{|k|^2}|\hat\theta|^2
\le-2\min\{\nu k_1^2,\eta k_2^2\}\,E_k .
\]
Since $k_1^2\ge1$ and $k_2^2\ge1$, the minimum is at least $\min\{\nu,\eta\}$, and
the bound follows by Gr\"onwall's inequality. If $k_1=0$ and $k_2\neq0$, then
$\hat u_2=-ik_1\hat\omega/|k|^2=0$ by the Biot--Savart law, and the second
equation of \eqref{lin_mode} reduces to
$d\hat\theta/dt=-\eta k_2^2\hat\theta$.

(b) For $k_2=0$ and $k_1\neq0$ the matrix of \eqref{lin_mode} is
$\bigl(\begin{smallmatrix}-\nu k_1^2&1\\-1&0\end{smallmatrix}\bigr)$, with characteristic equation
$\lambda^2+\nu k_1^2\lambda+1=0$. For $\nu k_1^2>2$ its roots
$\lambda_\pm=\big(-\nu k_1^2\pm\sqrt{\nu^2k_1^4-4}\big)/2$ are real, distinct and
negative, which gives the representation of $\hat\theta$. Multiplying the
numerator and the denominator of $\lambda_+$ by $\nu k_1^2+\sqrt{\nu^2k_1^4-4}$
gives the stated formula, and $0\le\sqrt{\nu^2k_1^4-4}\le\nu k_1^2$ gives the
bounds on $|\lambda_+|$ and $\lambda_-\le-\nu k_1^2/2$.
\end{proof}

\bibliographystyle{siamplain}
\bibliography{refer030}

\end{document}